\documentclass[letterpaper, 10 pt, conference]{ieeeconf}  

\IEEEoverridecommandlockouts                              
\usepackage{algorithmic,algorithm}
\usepackage{rotating}
\usepackage {amssymb}
\usepackage{graphicx,caption,tabularx,multicol,multirow,color}
\usepackage {amsmath}
\usepackage{bbm}
\usepackage{xcolor}
\usepackage{slashbox}
\usepackage{pifont}
\usepackage{caption}
\newtheorem{defn}{Definition}
\newtheorem{thm}{Theorem}
\newtheorem{lemma}{Lemma}
\newtheorem{col}{Corollary}
\newtheorem{pre}{Proposition}
\newtheorem{rem}{Remark}
\newtheorem{assumption}{Assumption}
\usepackage{mathrsfs}
\usepackage{graphicx}
\usepackage{hyperref}
\usepackage{array}
\usepackage{float}
\usepackage{comment}
\newcommand{\tr}[1]{\text{tr}\!\left(#1\right)}
\newcommand{\EX}{\mathbb{E} } 
\newcommand{\nm}[1]{{\color{black}#1}}
\newcommand{\cmark}{\ding{51}}%
\newcommand{\xmark}{\ding{55}}%

\title{\LARGE \bf A First-order Method for Monotone Stochastic Variational Inequalities on Semidefinite Matrix Spaces}

\author{Nahidsadat Majlesinasab$^{1}$, Farzad Yousefian$^{2}$, and Mohammad Javad Feizollahi$^{3}$
\thanks{$^{1}$Nahidsadat Majlesinasab is with the Department of Industrial Engineering and Management, Oklahoma State University,
         Stillwater, OK 74075, USA.
        {\tt\small nahid.majlesinasab@okstate.edu}}%
\thanks{$^{2}$ Farzad Yousefian is with Faculty of Industrial Engineering and Management, Oklahoma State University,
        Stillwater, OK 74078, USA.
        {\tt\small farzad.yousefian@okstate.edu}}%
\thanks{$^{3}$Mohammad Javad Feizollahi is with with Faculty of Robinson College of Business, Georgia State University, Atlanta, GA 30302, USA.
        {\tt\small mfeizollahi@gsu.edu}}%
}

\begin{document}

\maketitle
\thispagestyle{empty}
\pagestyle{empty}

\begin{abstract}
Motivated by multi-user optimization problems and non-cooperative Nash games in stochastic regimes, we consider stochastic variational inequality (SVI) problems on matrix spaces where the variables are positive semidefinite matrices and the mapping is merely monotone. Much of the interest in the theory of variational inequality (VI) has focused on addressing VIs on vector spaces.\nm{ Yet, most existing methods either rely on strong assumptions, or require a two-loop framework where at each iteration, a projection problem, i.e., a semidefinite optimization problem needs to be solved.} Motivated by this gap, we develop a stochastic mirror descent method where we choose the distance generating function to be defined as the quantum entropy. This method is a \nm{ single-loop} first-order method in the sense that it only requires a gradient-type of update at each iteration. The novelty of this work lies in the convergence analysis that is carried out through employing an auxiliary sequence of stochastic matrices. Our contribution is three-fold: (i) under this setting and employing averaging techniques, we show that the iterate generated by the algorithm converges to a weak solution of the SVI; (ii) moreover, we derive a convergence rate in terms of the expected value of a suitably defined gap function; (iii) we implement the developed method for solving a multiple-input multiple-output multi-cell cellular wireless network composed of seven hexagonal cells and present the numerical experiments supporting the convergence of the proposed method. 
\end{abstract}

\section{Introduction} \label{sec:int}
 Variational inequality problems first introduced in the 1960s have a wide range of applications arising in engineering, finance, and economics (cf. \cite{facchinei2007finite}) and are strongly tied to the game theory.
VI theory provides a tool to formulate different equilibrium problems and analyze the problems in terms of existence and uniqueness of solutions, stability and sensitivity analysis. In mathematical programming, VIs encompass problems such as systems of nonlinear equations, optimization problems, and complementarity problems to name a few 
\cite{scutari2010convex}. In this paper, we consider stochastic variational inequality problems where the variable $ X$ is a positive semidefinite matrix. 
Given a set ${\mathcal X}=\{  X \in \mathbb{S}^+_{n}, \tr{X}=1\}$, and a mapping $F:\mathbf{\mathcal X} \to \mathbb{R}^{n\times n}$, a VI problem denoted by VI$(\mathbf{\mathcal X}, F)$ seeks a positive semidefinite matrix $  X^* \in \mathbf{\mathcal X}$ such that 
\begin{align}
\label{eq:VI2}
\tr{F(  X^*)(  X-  X^*)} \geq 0, \quad \text{for all} ~  X \in    {\mathcal X}.
\end{align}
In particular, we study VI($\mathbf{\mathcal X}, F$) where $F(  X)=\EX[\Phi(  X,\xi(w))]$, i.e., the mapping $F$ is the expected value of a stochastic mapping $\Phi:\mathbf{\mathcal X}\times \mathbb{R}^{d} \to \mathbb{R}^{n\times n}$ where the vector $\xi:\Omega \to \mathbb{R}^d$ is a random vector associated with a probability space represented by $(\Omega, \mathcal{F},\mathbb{P})$. Here, $\Omega$ denotes the sample space, $\mathcal{F}$ denotes a $\sigma$-algebra on $\Omega$, and $\mathbb{P}$ is the associated \nm{probability measure.} Therefore, $X^* \in \mathbf{\mathcal X}$ solves VI($\mathbf{\mathcal X}, F$) if 
\begin{align}
\label{eq:VI}
	\tr{\EX[\Phi(  X^*,\xi(w))](  X-  X^*)} \geq 0,~\text{for all}~X \in {\mathcal X}.
\end{align}
Throughout, we assume that $\EX[\Phi(  X^*,\xi(w))]$ is well-defined (i.e., the expectation is finite).
\nm{\subsection{Motivating Example}}
\label{sec:motiv}
A non-cooperative game involves a number of decision makers called players who have conflicting interests and each tries to minimize/maximize his own payoff/utility function. Assume there are $N$ players each controlling a positive semidefinite matrix variable $  X_i$ which belongs to the set of all possible actions of the player $i$ denoted by $\mathcal X_i$. Let us define $  X_{-i}:\triangleq(  X_1,...,  X_{i-1},  X_{i+1},...,  X_N)$ as the feasible actions of other players. Let the payoff function of player $i$ be quantified by $f_i(X_i,X_{-i})$. Then, each player $i$ needs to solve the following semidefinite optimization problem 
\begin{align}
\label{eqn:problem1} 
\underset{X_i \in \mathcal X_i} {\text{minimize}}\quad f_i(X_i,X_{-i}).
\end{align}
A solution $  X^*=\left(  X^*_1,\ldots,   X^*_N\right)$ to this game called a Nash equilibrium is a feasible strategy profile such that $f_i(  X_i^*,  X_{-i}^*)\leq f_i(  X_i,  X^*_{-i})$, 
for all \nm{$X_i \in \mathcal X_i=\{X_i| X_i \in \mathbb S^+_{n_i}, \tr{X_i}=1\}$}, $i=1,\ldots,N$. \nm{As we discuss in Lemma \ref{lemma:nash}, }the optimality conditions of the above Nash game can be formulated as a VI$(\mathbf{\mathcal X}, F)$ where $\mathcal X:\triangleq \{  X|  X=\text{diag} (  X_1,\cdots,  X_N), \:   X_i\in \mathcal X_i,~\text{for all}~i=1,\ldots,N\}$ and $F(X):\triangleq \text{diag}(\nabla_{  X_1} f_1(  X),\cdots,\nabla_{  X_N} f_N(  X))$.\\
Problem \eqref{eqn:problem1} has a wide range of applications in wireless communications and information theory. Here we discuss a communication network example.

\textit{Wireless Communication Networks:} A wireless network is founded on transmitters that generate radio waves and receivers that detect radio waves.
To enhance the performance of the wireless transmission system, multiple antennas can be used to transmit and receive the radio signals. This system is called multiple-input multiple-output (MIMO) which provides high spectral efficiency in single-user wireless links without interference \cite{foschini1998limits}. Other MIMO systems include MIMO broadcast channels and MIMO multiple access channels, where there are multiple users (players) that mutually interfere. In these systems players either share the same transmitter or the same receiver. 
Recently, there has been much interest in MIMO systems under uncertainty when the state channel information is subject to measurement errors, delays or other imperfections \cite{mertikopoulos2017distributed}.  
Here, we consider the throughput maximization problem in multi-user MIMO networks under feedback errors and uncertainty. In this problem, we have $N$ MIMO links where each link $i$ represents a pair of transmitter-receiver with $m_i$ antennas at the transmitter and $n_i$ antennas at the receiver. We assume each of these links is a player of the game. 
Let $\mathbf  x_i \in \mathbb C^{m_i}$ and $\mathbf y_i \in \mathbb C^{n_i}$ denote the signal transmitted from and received by the $i$th link, respectively. The signal model can be described by $\mathbf{y}_i={H}_{ii}\mathbf{x}_i+\sum\nolimits_{j \ne i}{H}_{ji}\mathbf{x}_j+\mathbf{\epsilon}_i$,
where ${H}_{ii} \in \mathbb C^{n_i \times m_i}$ is the direct-channel matrix of link $i$, ${H}_{ji} \in \mathbb C^{n_i \times m_j}$ is the cross-channel matrix between transmitter $j$ and receiver $i$, and $\mathbf {\epsilon}_i \in \mathbb C^{n_i}$ is a zero-mean circularly symmetric complex Gaussian noise vector with the covariance matrix $\mathbf I_{m_i}$ \cite{mertikopoulos2016learning}. 
 The action for each player is the transmit power, meaning that each transmitter $i
$ wants to transmit at its maximal power level to improve its performance. However, doing so increases the overall interference in the system, which in turn, adversely impacts the performance of all involved transmitters and presents a conflict. 
It should be noted that we treat the interference generated by other users as an additive noise. Therefore, $\sum_{j \ne i}{H}_{ji}\mathbf{x}_j$ 
represents the multi-user interference (MUI) received by $i$th player and generated by other players. 
Assuming the complex random vector $\mathbf{x}_i$ follows a Guassian distribution, transmitter $i$ controls its input signal covariance matrix $  X_i:\triangleq\EX[\mathbf{x}_i\mathbf{x}_i^\dag]$ subject to two constraints: first the signal covariance matrix is positive semidefinite and second each transmitter's maximum transmit power is bounded by a positive scalar $p$.  
Under these assumptions, each player's achievable transmission throughput for a given set of players' covariance matrices $X_1,\ldots,X_N$ is given by
\begin{align}
\label{eq:game}
	R_i(  X_i,  X_{-i})&=\log \det\left(\mathbf I_{n_i}+\sum\nolimits_{j=1}^N H_{ji}  X_j  H_{ji}^\dagger\right)\nonumber\\&-\log \det(W_{-i}),
\end{align}
 where $W_{-i}=\mathbf I_{n_i}+\sum_{j \ne i}H_{ji}X_j H_{ji}^\dagger$ is the MUI-plus-noise covariance matrix at receiver $i$ \cite{telatar1999capacity}. The goal is to solve 
\begin{align}
\label{eq:Rgame} 
\underset{X_i \in \mathcal X_i} {\text{maximize}}\quad R_i(  X_i,  X_{-i}),
\end{align}
for all $i$, where $\mathcal X_i=\{X_i:X_i\succeq 0$, $\tr{X_i}\leq p\}$.
\subsection{Existing methods}
Our primary interest in this paper lies in solving SVIs on semidefinite matrix spaces. Computing the solution to this class of problems is challenging mainly due to the presence of uncertainty and the semidefinite solution space. In what follows, we review some of the methods in addressing these challenges. More details are presented in Table \ref{tbl:comp}.
\begin{table*}[ht]
\caption{Comparison Of Schemes}
\label{tbl:comp}
	\centering
\begin{tabular}{|c|c|c|c|c|c|c|c|}
	\hline
	 Reference & Problem & Characteristic & Assumptions & Space & Scheme & Single-loop & Rate \\ 
	\hline
	 Jiang and Xu \cite{jiang2008stochastic} & VI & Stochastic & SM,S & Vector & SA & \xmark & $-$ \\
	\hline 
	Juditsky et al. \cite{juditsky2011solving} & VI & Stochastic & MM,S/NS & Vector & Extragradient SMP  & \xmark & ${\cal O} \left({1}/{t}\right)$ \\
	\hline 
	Lan et al. \cite{lan2011primal} & Opt & Deterministic & C,S/NS & Matrix & \begin{minipage}{2.3cm} Primal-dual \\ Nesterov's methods \end{minipage}& \xmark & ${\cal O} \left({1}/{t}\right)$  \\
	\hline 
		Mertikopoulos et al. \cite{mertikopoulos2012matrix} & Opt & Stochastic &  C,S  & Matrix & Exponential Learning & \cmark & $e^{-\alpha t}\nm{(\alpha>0)}$ \\
	\hline 
	Hsieh et al. \cite{hsieh2013big} & Opt & Deterministic & NS,C & Matrix & BCD & \xmark & superlinear \\
	\hline 
	Koshal et al. \cite{koshal2013regularized} & VI & Stochastic & MM,S & Vector & Regularized Iterative SA & \xmark & $-$ \\
	\hline 
		Yousefian et al. \cite{yousefian2016stochastic} & VI & Stochastic & PM,S & Vector & Averaging B-SMP & \xmark & ${\cal O} \left({1}/{t}\right)$ \\
	\hline 
	Yousefian et al. \cite{yousefian2017smoothing}  & VI & Stochastic & MM,NS & Vector & Regularized Smooth SA & \xmark & ${\cal O} \left({1}/{\sqrt{t}}\right)$ \\
	\hline 
	Mertikopoulos  et al. \cite{mertikopoulos2017distributed} & VI & Stochastic & SM,S & Matrix & Exponential Learning & \cmark &${\cal O} \left({1}/{Bt}\right)$   \\
	\hline 	
Necoara et al. \cite{necoara2017complexity} & Opt & Deterministic & C,S/NS & Vector & Inexact Lagrangian & \xmark & ${\cal O} \left({1}/{t^{1.5}}\right)$  \\
	\hline 
	\textbf{Our work} & VI & Stochastic & MM, NS & Matrix & AM-SMD & \cmark & ${\cal O} \left({1}/{\sqrt{t}}\right)$ \\
	\hline 
\end{tabular}
\begin{minipage}{15cm}%
\vspace{0.08in}
 SM: \textit{strongly monotone mapping},\quad MM: \textit{merely monotone mapping}, \quad PM: \textit{psedue-monotone mapping},\quad S: \textit{smooth function} \\ NS: \textit{nonsmooth function},\quad C: \textit{convex},\quad Opt: \textit{optimzation problem}, \quad \nm{B: strong stability parameter }
  \end{minipage}%
\end{table*}

\textbf{Stochastic approximation (SA) schemes:} \nm{The} SA method was first developed \nm{in} \cite{robbins1951stochastic} and has been very successful in solving optimization and equilibrium problems with uncertainties. \nm{Jiang and Xu \cite{jiang2008stochastic}} appear amongst the first who applied SA methods to \nm{address} SVIs. 
In recent years, prox generalization of SA methods were developed for solving stochastic optimization problems \cite{nemirovski2009robust,majlesinasab2017optimal} and VIs. 
%
The monotonicity of the gradient mapping plays an important role in the convergence analysis of this class of solution methods. The extragradient method which relies on weaker assumptions, i.e., pseudo-monotone mappings to address VIs was developed \nm{in} \cite{korpelevich1977extragradient}, but this method requires two projections per iteration. \nm{Dang and Lan \cite{dang2015convergence} }developed a non-Euclidean extragradient method to address generalized monotone VIs. 
The prox generalization of the extragradient schemes to stochastic settings were developed \nm{in} \cite{juditsky2011solving}. 
Averaging techniques first introduced \nm{in} \cite{Polyak92} proved successful in increasing the robustness of the SA method. In vector spaces equipped with non-Euclidean norms, \nm{Nemirovski et al. \cite{nemirovski2009robust}} developed the stochastic mirror descent (SMD) method for solving nonsmooth stochastic optimization problems. 
\nm{While SA schemes and their prox generalization can be applied directly to solve problems with semidefinite constraints, they result in a two-loop framework and require projection onto a semidefinite cone by solving an optimization problem at each iteration which increases the computational complexity.} 

\textbf{Exponential learning methods:} Optimizing over sets of positive semidefinite matrices is more challenging than vector spaces because of the form of the problem constraints. 
In this line of research, \nm{an approach based on matrix exponential learning (MEL) is proposed in \cite{mertikopoulos2012matrix}} to solve the power allocation problem in MIMO multiple access channels. MEL is an optimization algorithm applied to positive definite nonlinear problems and has strong ties to mirror descent methods. MEL makes the use of quantum entropy as the distance generating function. Later, the convergence analysis of MEL \nm{is provided in \cite{mertikopoulos2016learning} and its robustness w.r.t. uncertainties is shown}. \nm{In} \cite{yu2016dynamic}, single-user MIMO throughput maximization problem \nm{is addressed which is an optimization problem not a Nash game.
In the multiple channel case, an optimization problem can be derived that makes the analysis much easier. However, there are some practical problems that cannot be treated as an optimization problem such as multi-user MIMO maximization discussed earlier. In this regard, \cite{mertikopoulos2017distributed} proposed an algorithm relying on MEL for solving N-player games under feedback errors and presented its convergence to a stable Nash equilibrium under a strong stability assumption. However, in most applications including the game \eqref{eq:game}
 the mapping does not satisfy this assumption.}

\textbf{Semidefinite and cone programming:}
Sparse inverse covariance estimation (SICE) is a procedure which improves the stability of covariance estimation by setting a certain number of coefficients in the inverse covariance to zero.
\nm{Lu} \cite{lu2010adaptive} developed two first-order methods including the adaptive spectral projected gradient and the adaptive Nesterov's smooth methods to solve the large scale covariance estimation problem. 
 In this line of research, a block coordinate descent (BCD) method with a superlinear convergence rate is proposed \nm{in \cite{hsieh2013big}}. 
In conic programming with complicated constraints, many first-order methods are combined with duality or penalty strategies \cite{lan2011primal,necoara2017complexity}. These methods are projection based and do not scale with \nm{the} problem size. 

Much of \nm{the} interest in \nm{the VI regime} has focused on addressing VIs on vector spaces. 
Moreover, in the literature of semidefinite programming, most of the methods address deterministic semidefinite optimization. Yet, there are many stochastic systems such as wireless communication systems that can be modeled as positive semidefnite Nash games. 
In this paper, we consider SVIs on matrix spaces where the mapping is merely monotone. Our main contributions are as follows:

\textit{(i) Developing an averaging matrix stochastic mirror descent (AM-SMD) method:} We develop an SMD method where we choose the distance generating function to be defined as the quantum entropy following \cite{tsuda2005matrix}. It is a first-order method in the sense that only a gradient-type of update at each iteration is needed. \nm{The algorithm does not need a projection step at each iteration since it provides a closed-form solution for the projected point}. To improve the robustness of the method for solving SVI, we use the averaging technique.
Our work is an improvement to MEL method \cite{mertikopoulos2017distributed} and is motivated by the need to weaken the strong stability (monotonicity) requirement on the mapping.  
\nm{The main novelty of our work lies in the convergence analysis in absence of strong monotonicity where we introduce an auxiliary sequence and we are able to establish convergence to a weak solution of the SVI.} Then, we derive a convergence rate of $\mathcal O(1/\sqrt{T})$ in terms of the expected value of a suitably defined gap function. \nm{To clarify the distinctions of our contributions, we prepared Table \ref{tbl:comp} where we summarize the differences between the existing methods and our work.} 

\textit{(ii) Implementation results:} We present the performance of the proposed AM-SMD method applied on the throughput maximization problem in wireless multi-user MIMO networks. Our results indicate the robustness of the AM-SMD scheme with respect to problem parameters and uncertainty. Also, it is shown that the AM-SMD outperforms both non-averaging M-SMD and MEL \nm{\cite{mertikopoulos2017distributed}}.

The paper is organized as follows. In Section \ref{sec:algorithm}, we state the assumptions on the problem and outline our AM-SMD algorithm. Section \ref{sec:convergence} contains the convergence analysis and the rate derived for the AM-SMD method. We report some numerical results in Section \ref{sec:num} and conclude in Section \ref{sec:conclusion}.

\textbf{Notation.} Throughout, we let $\mathbb S_n$ denote the set of all $n \times n$ symmetric matrices and $\mathbb S_n^+$ the cone of all positive semidefinite matrices. We define $\mathscr X:=\{X\in \mathbb S_n^+:\tr{X} \leq 1\}$. The mapping $F:\mathcal X \to \mathbb R^{n \times n}$ is called monotone if for any $X,Y \in \mathcal X$, we have $\tr{(  X-  Y)(  F (  X)-  F (  Y))}\geq 0$. Let $[A]_{uv}$ denote the elements of matrix $A$ and $\mathbb C$ denote the set of complex numbers. The norm $\Vert A \Vert_2$ denotes the spectral norm of a matrix $A$ being the largest singular value of $A$. The trace norm of a matrix $A$ denoted by $\tr{A}$ is the sum of singular values of the matrix. Note that spectral and trace norms are dual to each other \cite{fazel2001rank}. 
We use $\text{SOL}({\mathcal X}, \mathit F)$ to denote the set of solutions to VI($\mathbf{\mathcal X},F$). 
\section{Algorithm outline}\label{sec:algorithm}
In this section, we present the AM-SMD scheme for solving \eqref{eq:VI}. 
Suppose $\omega: \text{dom}(\omega) \to \mathbb R$ is a strictly convex and differentiable function, where $\text{dom}(\omega) \subseteq \mathbb{R}^{n\times n}$, and let $X,Y\in \mathbb{R}^{n\times n}$. Then, Bregman divergence between $X$ and $Y$ is defined as 
\nm{
\begin{align*}
	D(  X,  Y):=\omega(  X)-\omega(  Y)-\tr{(  X-  Y)\nabla\omega(  Y)^T}.
\end{align*}
In what follows, our choice of $\omega$ is the quantum entropy \cite{vedral2002role}, 
\begin{equation}
\label{eq:entropy}
\omega(X) = \left\{
\begin{array}{rl}
\tr{X\log X-X}&~\text{if }\quad X\in \mathcal X,\\
+\infty \quad \quad \quad \quad \quad \quad & \quad \text{otherwise}.
\end{array} \right.
\end{equation}
The Bregman divergence corresponding to the quantum entropy is called von Neumann divergence and is given by
\begin{align}
	D(X,Y)=\tr{X\log X- X\log Y}
\end{align}
}
\cite{tsuda2005matrix}. In our analysis, we use the following property of $\omega$.
\begin{lemma}
\label{lm:strong} (\cite{yu2013strong})
The quantum entropy $\omega : \mathscr X\to \mathbb R$ is strongly convex with modulus 1 under the trace norm.
\end{lemma}
Since $ \mathcal X\subset\mathbf{\mathscr X}$, the quantum entropy $\omega:{\mathcal X}\to \mathbb R$ is also strongly convex with modulus 1 under the trace norm.

\nm{Next, we address the optimality conditions of a matrix constrained optimization problem as a VI which is an extension of Prop. 1.1.8 in \cite{bertsekas2009convex}.}
\begin{lemma} \label{lemma:optimality}
Let $\mathcal X \subseteq \mathbb R^{n \times n}$ be a nonempty closed convex set, and let $f:\mathbb R^{n \times n}\to \mathbb R$ be a differentiable convex function. Consider the optimization problem
\begin{align}
\label{eqn:problem2} 
\underset{\widetilde{X} \in \mathcal X} {\text{minimize}}\quad f( \widetilde{  X}).
\end{align}
A matrix $\widetilde{  X}^*$ is optimal to problem \eqref{eqn:problem2} iff $\widetilde{  X}^* \in \mathcal X$ and $\tr{\nabla^Tf(\widetilde{  X}^*)(  Z-\widetilde {  X}^*)} \geq 0$, for all $Z\in \mathcal X$.
\end{lemma}
\begin{proof} 
($\Rightarrow$) Assume $\widetilde{  X}^*$ is optimal to problem \eqref{eqn:problem2}. Assume by contradiction, there exists some $\hat{  Z} \in \mathcal X$ such that $\tr{\nabla_{\widetilde{  X}}^Tf(\widetilde{  X}^*)(\hat{  Z}-\widetilde {  X}^*)} < 0$. Since $f$ is continuously differentiable, by the first-order Taylor expansion, for all sufficiently small $0<\alpha<1$, we have 
\begin{align*}
&f(\widetilde{  X}^*+\alpha (\hat{  Z}-\widetilde {  X}^*))= f({  X}^*) +\\
 &\tr{\nabla_{\widetilde{  X}}^Tf(\widetilde{  X}^*) (\hat{  Z}-\widetilde {  X}^*)}+ o(\alpha) < f({  X}^*),
\end{align*}
following the hypothesis $\tr{\nabla_{\widetilde{  X}}^Tf(\widetilde{  X}^*)(\hat{  Z}-\widetilde {  X}^*)} < 0$. Since $\mathcal X$ is convex and $  X^*,~\hat{  Z} \in \mathcal X$, we have $\widetilde{  X}^*+\alpha (\hat{  Z}-\widetilde {  X}^*) \in \mathcal X$ with smaller objective function value than the optimal matrix $\widetilde{  X}^*$. This is a contradiction. Therefore, we must have $\tr{\nabla_{\widetilde{  X}}^T f(\widetilde{  X}^*)(  Z-\widetilde {  X}^*)} \geq 0$ for all $  Z\in \mathcal X$. 
\\($\Leftarrow$) Now suppose that  $\widetilde{  X}^* \in \mathcal X$ and for all $   Z\in \mathcal X$, $\tr{\nabla_{\widetilde{  X}}^T f(\widetilde{  X}^*)(  Z-\widetilde {  X}^*)} \geq 0$. Since $f$ is convex 
, we have
\begin{align*}
	f(\widetilde {  X}^*) +\tr{\nabla_{\widetilde{  X}}^T f(\widetilde{  X}^*)(  Z-\widetilde {  X}^*)} \leq f(  Z),
\end{align*}
for all $Z\in \mathcal X$ which implies
\begin{align*}
 f(  Z)-f(\widetilde {  X}^*) \geq 	\tr{\nabla_{\widetilde{  X}}^T f(\widetilde{  X}^*)(  Z-\widetilde {  X}^*)} \geq 0,
\end{align*}
where the last inequality follows by the hypothesis. Since $\widetilde {  X}^* \in \mathcal X$, it follows that $\widetilde {  X}^*$ is optima. 
\end{proof}
\nm{The next Lemma shows a set of sufficient conditions under which a Nash equilibrium can be obtained by solving a VI.}
\begin{lemma} \label{lemma:nash} 
[Nash equilibrium]
Let $\mathcal X_i \in \mathbb S_{n_i}$ be a nonempty closed convex set and $f_i(  X_i,  X_{-i})$ be a differentiable convex function in $  X_i$ for all $i=1,\cdots,N$, where $  X_i \in \mathcal X_i$ and $  X_{-i} \in \prod_{j\ne i} {\mathcal X_j}$. Then, $  X^*\triangleq \text{diag}(  X_1^*,\cdots,  X_N^*)$ is a Nash equilibrium (NE) to game \eqref{eqn:problem1} if and only if $  X^*$ solves VI($\mathbf{\mathcal X}, F$), where
\begin{align}
\label{Fdefinition}
&F(X):\triangleq \text{diag}(\nabla_{  X_1} f_1(  X),\cdots,\nabla_{  X_N} f_N(  X)),	\\
\label{xdefinition}
&\mathcal X:\triangleq \{  X|  X=\text{diag} (  X_1,\cdots,  X_N), \:   X_i\in \mathcal X_i,~\text{for all} ~ i\}.
\end{align}
\end{lemma}
\begin{proof}
First, suppose $  X^*$ is an NE to game \eqref{eqn:problem1}. We want to prove that $  X^*$  solves VI($\mathbf{\mathcal X}, F$), i.e, $\tr{F(  X^*)^T(Z-X^*)} \geq 0$, for all $Z\in \mathcal X$.
By optimality conditions of optimization problem \eqref{eqn:problem1} and from Lemma \ref{lemma:optimality}, we know $X^*$ is an NE if and only if $\tr{\nabla_{X_i}^T f_i(X^*)(Z_i-X_i^*)} \geq 0$ for all $Z_i\in \mathcal X_i$ \nm{and all} $i=1,\ldots,N$.
Then, we \nm{obtain} for all $i=1,\cdots,N$ 
\begin{align}
\label{eq:lemmfdefinition-2}
&\tr{\nabla_{X_i}^T f_i(  X^*)(  Z_i-  X_i^*)}=\notag\\&\sum_{u}\sum_{v}[\nabla_{X_i} f_i(  X^*)]_{uv}[Z_i-  X_i^*]_{uv}\geq 0.
\end{align}
Invoking the definition of mapping $F$ given by \eqref{Fdefinition} and \nm{from} \eqref{eq:lemmfdefinition-2}, we have 
$\tr{F(  X^*)^T(  Z-  X^*)}=\sum_{i,u,v}[\nabla_{X_i} f_i(  X^*)]_{uv}[Z_i-  X_i^*]_{uv}\geq 0.$
From the definition of VI($\mathbf{\mathcal X}, F$) and relation \eqref{eq:VI2}, we conclude that $  X^*\in \text{SOL}({\mathcal X}, \mathit F)$. Conversely, suppose $  X^* \in \text{SOL}({\mathcal X}, \mathit F)$. Then, $\tr{F(  X^*)^T(  Z-  X^*)} \geq 0, \text{for all} \:   Z \in \mathcal X$. Consider a fixed $i \in \{1,\ldots,N\}$ and a matrix $\bar{  Z} \in \mathcal X$ given by \eqref{xdefinition} such that the only difference between $  X^\ast$ and $\bar{  Z}$ is in $i$-th block, i.e. 
\begin{align*}
\bar{  Z}= \text{diag}\left(\left[  X_1^*\right],\ldots,\left[  X_{i-1}^*\right],\left[  Z_i\right],\left[  X_{i+1}^*\right],\ldots,\left[  X_{N}^*\right] \right),
\end{align*}
where $  Z_i$ is an arbitrary matrix in $\mathcal X_i$. Then, we have
\begin{align}
\label{eq:z-x}
\bar{  Z}-  X^*= \text{diag}\left(\mathbf 0_{n_1\times n_1},\ldots,\left[  Z_{i} -  X_i^*\right],\ldots,\mathbf 0_{n_N\times n_N} \right).
\end{align}
Therefore, substituting $\bar{  Z}-  X^*$ by term \eqref{eq:z-x}, we obtain 
\begin{align*}
	\tr{F(  X^*)^T(\bar{  Z}-  X^*)}=\sum_{u}\sum_{v}[\nabla_{  X_i} f_i(  X^*)]_{uv}\nonumber \\
	\times[(  Z_i-  X_i^*)]_{uv}=\tr{\nabla_{  X_i}^T f_i(  X^*)(  Z_i-  X_i^*)}\geq 0.
\end{align*}
Since $i$ was chosen arbitrarily, $\tr{\nabla_{  X_i}^T f_i(  X^*)(  Z_i-  X_i^*)}\geq 0$ for any $i=1,...,N$. Hence, by applying Lemma \ref{lemma:optimality} we conclude that $  X^*$ is a Nash equilibrium to game \eqref{eqn:problem1}. 
\end{proof}
\nm{Algorithm \ref{alg} presents the outline of the AM-SMD method. At each iteration $t$, first, using an oracle, a realization of the stochastic mapping $F$ is generated at $X_t$, denoted by $\Phi(  X_t, \xi_t)$. Next, a matrix $Y_t$ is updated using \eqref{eq:y}. Here $\eta_t$ is a non-increasing step-size sequence. Then, $Y_t$ will be projected onto set $\mathcal X$ using the closed-form solution \eqref{eq:x}. Then the averaged sequence $\overline{X}_{t+1}$ is generated using relations $\eqref{eq:gamma}$.}
Next, we state the main assumptions. Let us define the stochastic error at iteration $t$ as
\begin{align}
\label{eq:z-definition}
	   Z_{t} :\triangleq \Phi(  X_t, \xi_t)- F(  X_t) \quad \text{for all}  \quad t\geq 0.
\end{align}
Let $\mathcal F_t$ denote the history of the algorithm up to time $t$, i.e., $\mathcal F_t=\{  X_0, \xi_0,\ldots,\xi_{t-1}\}$ for $t\geq 1$ and $\mathcal F_0=\{  X_0\}$. 
\begin{assumption} Let the following hold:
\label{ass:boundphi}
\begin{itemize}
\item [(a)] The mapping $F(  X)=\EX[\Phi(  X_t,\xi_t)]$ is monotone and continuous over the set $\mathbf{\mathcal X}$. 
\item [(b)] The stochastic mapping $\Phi(  X_t,\xi_t)$ has a finite mean squared error, i.e, there exist some $C>0$ such that $\EX[\Vert\Phi(  X_t,\xi_t)\Vert^2_2|\mathcal F_t] \leq C^2$.\nm{ (Under this assumption, the mean squared error of the stochastic noise is bounded.)}
\item [(c)] The stochastic noise $  Z_t$ has a zero mean, i.e., $\EX[  Z_t|\mathcal F_t]=0$ for all $t\geq 0$.
\end{itemize}
\end{assumption} 
\begin{algorithm}
 \caption{Averaging Matrix Stochastic Mirror Descent (AM-SMD)}
\label{alg}
\begin{algorithmic}
     \STATE \textbf{initialization}: Set $  Y_0:=  I_n/n$, a stepsize $\eta_0 > 0$, $\Gamma_0=\eta_0$ and let $  X_0 \in \mathbf{\mathcal X}$ and $\overline{  X}_{0}=  X_0$. 
      \FOR {$t=0,1,...,T-1$}	
				\STATE Generate $\xi_t$ as realizations of the random matrix $\xi$ and \nm{evaluate} the mapping $\Phi(  X_t, \xi_t)$. Let 
				\begin{align} \label{eq:y}
				&Y_{t+1} :=   Y_t - \eta_{t}  \Phi(X_t, \xi_t),\\
				\label{eq:x}
				&X_{t+1}:=\displaystyle\frac{\exp(  Y_{t+1}+  I_n)}{\tr{\exp(  Y_{t+1}+  I_n)}}.
				\end{align}
					\STATE Update $\Gamma_t$ and $\overline{  X}_{t}$ using the following recursions:
					\begin{align}\label{eq:gamma}
						 &\Gamma_{t+1}:=\Gamma_{t}+\eta_{t+1},~\overline{X}_{t+1}:=\frac{\Gamma_t\overline{X}_t+\eta_{t+1}{X}_{t+1}}{\Gamma_{t+1}}.
					\end{align}
\ENDFOR 
\STATE Return $\overline{X}_{T}$.
\end{algorithmic}
\end{algorithm}
\section{Convergence analysis} \label{sec:convergence}
\nm{In this section, our interest lies in analyzing the convergence and deriving a rate statement for the sequence generated by the AM-SMD method.}
\nm{Note that a solution of VI(${\mathcal X}, F$) is also called a strong solution. Next, we define a weak solution which is considered to be a counterpart of the strong solution.}
\begin{defn} \label{def:stable}(Weak solution) The matrix ${X}^*_w \in \mathcal X$ is called a weak solution to VI($\mathbf{\mathcal X}, F$) if it satisfies 
$\tr{F(  X)(  {X}-  {X}^*_w)} \geq 0$, for all $X \in \mathbf{\mathcal X}.$
\end{defn}
Let us denote $\mathcal X^\star_w$ and $\mathcal X^*$  the set of weak solutions and strong solutions to VI($\mathbf{\mathcal X}, F$), respectively.
\begin{rem}  Under Assumption \ref{ass:boundphi}(a), when the mapping $F$ is monotone, any strong solution of problem \eqref{eq:VI} is a weak solution, i.e., $\mathcal X^* \subseteq \mathcal X^\star_w$. 
Providing that $F$ is also continuous, the inverse also is true and a weak solution is a strong solution. Moreover, for a monotone mapping $F$ on a convex compact set e.g., $\mathcal X$, a weak solution always exists \cite{juditsky2011solving}. 
\end{rem}

\nm{Unlike optimization problems where the function provides a metric for distinguishing solutions, there is no immediate analog in VI problems. However, we use the following residual function associated with a VI problem.}
\begin{defn} ($G$ function) Define the following function $G: \mathcal X \to \mathbb R$ as
\label{def:gap}
\begin{align*}
	G(  {X})= \underset {  Z \in \mathcal X} {\sup}\tr{F(  Z)(  {X}-  {Z})}, \quad \text{for all}~   {X} \in  \mathcal X.
\end{align*}
\end{defn}
The next lemma provides some properties of the $G$ function.
\begin{lemma}
\label{lm:gap properties2}
The function $G({X})$ given by Definition \ref{def:gap} is a well-defined gap function, i.e, $(i)$ $G(  {X})\geq 0$ for all $  X \in \mathcal X$; $(ii)$ $  X^*_w$ is a weak solution to problem \eqref{eq:VI} iff $G(  {X}^*_w)=0$.
\end{lemma}
\begin{proof}
$(i)$ For an arbitrary $  X \in \mathcal X$, we have
\begin{align*}
		G(  {X})= \underset {  Z \in \mathcal X} {\sup}\tr{F(  Z)(  {X}-  {Z})} \geq \tr{F(  A)(  {X}-  {A})},
\end{align*}
for all ${A} \in  \mathcal X$. For $A=X$, the above inequality suggests that $G(  {X})\geq \tr{F(  X)(  {X}-  {X})}=0$ implying that the function $G(  {X})$ is nonnegative for all $  {X} \in  \mathcal X$. \\
$(ii)$ Assume $  X^*_w$ is a weak solution. By Definition \ref{def:stable}, $\tr{F(  X)(  {X}^*_w-  {X})} \leq 0$, for all $X \in \mathbf{\mathcal X}$ which implies 
$G(  {X}^*_w)= \underset {  X \in \mathcal X} {\sup}\tr{F(  X)(  {X}^*_w-  {X})}	 \leq 0$.
On the other hand, from Lemma \ref{lm:gap properties2}$(i)$, we get $G(  {X}^*_w) \geq 0$. We conclude that $G(  {X}^*_w) = 0$ for any weak solution ${X}^*_w$. 
Conversely, assume that there exists an ${X}$ such that $G(  {X}) = 0$. Therefore, $\underset {  Z \in \mathcal X} {\sup}\tr{F(  Z)(  {X}-  {Z})}=0$ which implies $\tr{F(  Z)(  {Z}-  {X})} \geq 0$ for all $  Z \in \mathcal X$ implying ${X}$ is a weak solution.
   \end{proof}
\begin{rem} 
\label{lemma:average}
Assume the sequence $\eta_t$ is non-increasing and the sequence $\overline{X}_{t}$ is given by the recursive rules \eqref{eq:gamma} where $\Gamma_0=\eta_0$ and $\overline{  X}_{0}={  X}_{0}$. Then, using induction, it can be shown that $\overline{X}_{t}:=\sum_{k=0}^{t} \left(\frac{\eta_k}{\sum_{k'=0}^{t} \eta_{k'}}\right)  X_k$ for any $t\geq 0$.
\end{rem}
Next, we derive the conjugate of the quantum entropy and its gradient. 
\begin{pre} 
Let $  Y \in \mathbb S_n$ and $\omega(X)$ be defined as \eqref{eq:entropy}. Then, we have 
\begin{align}
\label{eq:conjugate}
&\omega^*(Y)=\log(\tr{\exp(Y+I_n)}),\\
\label{eq:gradient-omega-star}
&\nabla\omega^*(Y)=\frac{\exp(Y+I_n)}{\tr{\exp(Y+I_n)}}.
\end{align}
\end{pre}
\begin{proof}
$\omega$ is a lower semi-continuous convex function on the linear space of all symmetric matrices. The conjugate of function $\omega$ can be defined as 
\begin{align*}
	\label{eq:}
 &\omega^*({Y})=\sup\{\tr{  D  Y}-\omega(  D) : ~  D \in \mathcal X\}\\
&=\sup\{\tr{ D  Y}-\tr{  D \log   D-  D} :  D \in \mathcal X \}\\
&=-\inf\{-\underbrace{\tr{  D(  Y+  I_n)}+\tr{  D \log   D}:   D \in \mathcal X}_{\text{Term 1}}\}.
\end{align*}
The minimizer of the above problem is $\displaystyle  D= \frac{\exp(  Y+  I_n)}{\tr{\exp(  Y+  I_n)}}$ which is called the Gibbs state (see \cite{hiai2014introduction}, Example 3.29). We observe that $D$ is a positive semidefinite matrix with trace equal to one, implying that $D \in \mathcal X$. By plugging it into Term 1, we have \eqref{eq:conjugate}. The relation \eqref{eq:gradient-omega-star} follows by standard matrix analysis and the fact that $\nabla_  {Y}\tr{\exp(  Y)}=\exp(  Y)$ \cite{athans1965gradient}.
 \end{proof}
Throughout, we use the notion of Fenchel coupling \cite{mertikopoulos2016learning2}:
\begin{equation}
\label{eq:fenchel}
H(  {Q},  {Y})\triangleq\omega(  {Q})+\omega^*(  {Y})-\tr{  {Q}  {Y}},
\end{equation}
 which provides a proximity measure between $  {Q}$ and $\nabla \omega^*(  {Y})$ and is equal to the associated Bregman divergence between $  {Q}$ and $\nabla \omega^*(  {Y})$. \nm{We also make use of the following Lemma which is proved in Appendix.}  
\begin{lemma} (\cite{mertikopoulos2017distributed})
\label{pre:smoothstrong}
For all matrices $X\in \mathcal X$ and for all ${Y},{Z} \in \mathbb S_n$, the following holds
\begin{equation}  
\label{eq:smoothstrong}
H(  {X},  {Y}+  {Z})\leq H(  {X},  {Y})+ \tr{  Z(\nabla \omega^*(  {Y})-  {X})}+\Vert   Z\Vert^2_2.
\end{equation}
\end{lemma}
Next, we develop an error bound for the G function. For simplicity of notation we use $\Phi_t$ to denote $\Phi({X}_t,\xi_t)$.  
\begin{lemma}
\label{lemma:convergence}
Consider problem \eqref{eq:VI}. Let $  X \in \mathcal X$ and the sequence $\{\overline{  X}_t\}$ be generated by AM-SMD algorithm. Suppose Assumption \ref{ass:boundphi} holds. Then, for any $T \geq 1$, 
\begin{align}  
\label{eq:bound-1}
 \EX[G(\overline{  X}_T)] &\leq 2\left(\frac{\log(n)+\sum_{t=0}^{T-1}\eta_t^{2} C^2}{\sum_{t=0}^{T-1}\eta_t}\right).
\end{align}
\end{lemma}
  \begin{proof}
From the definition of $Z_t$ in relation \eqref{eq:z-definition}, the recursion in the AM-SMD algorithm can be stated as
\begin{equation}  
\label{eq:lem-Error bounds-1}
  {Y}_{t+1}=  {Y}_t-\eta_t (  F(X_t)+  Z_t).
\end{equation}
\nm{Consider \eqref{eq:smoothstrong}. From Algorithm \ref{alg} and \eqref{eq:gradient-omega-star}, we have $X_t=\nabla\omega^*(Y_t)$. Let $Y:=Y_t$ and $Z:=-\eta_t (F(X_t)+ Z_t)$.} From \eqref{eq:lem-Error bounds-1}, we obtain
\begin{align*}  
&H(  {X},  Y_{t+1})\leq H(  {X},  Y_{t})-\eta_t \tr{(  X_t-  X)(  F(  X_t)+  Z_t)}\\&+\eta_t^2\Vert   F(  X_t)+  Z_t\Vert^2_2.
\end{align*}
By adding and subtracting $\eta_t \tr{(X_t-X)F(X)}$, we get
\begin{align}
\label{eq:mon}  
&H(  {X},  Y_{t+1}) 
\leq H(  {X},  Y_{t})-\eta_t \tr{(  X_t-  X)  Z_t}\notag\\&-\eta_t \tr{(  X_t-  X)  F(  X)}+\eta_t^2\Vert   F(  X_t)+  Z_t\Vert^2_2,
\end{align}  
where we used the monotonicity of mapping $F$. 
Let us define an auxiliary sequence $U_{t}$ such that $  U_{t+1}:\triangleq  U_{t}+\eta_t  Z_{t}$, where $  U_0 =\mathbf I_n$ and define $  V_{t}:\triangleq   F(  U_{t})$. From \eqref{eq:mon}, invoking the definition of $  Z_{t}$ and by adding and subtracting $  V_t$, we obtain
\begin{align}  
\label{eq:13}
&\eta_t \tr{(X_t-X)F(X)} \leq H(X,Y_t)-H(X,Y_{t+1})+\notag\\
&\eta_t \tr{(V_t-X_t)Z_t}+\eta_t \tr{(X-V_t)Z_t}+\eta_t^2\Vert \Phi_t\Vert^2_2.
\end{align}
Then, we estimate the term $\eta_t \tr{(X-V_t)Z_t}$. By Lemma \ref{pre:smoothstrong} and setting $  Y:=  U_{t}$ and $  Z:=\eta_t  Z_t$, we get
\begin{align*}
\eta_t \tr{(X-V_t)Z_t}&\leq H({X},U_{t})-H({X},U_{t+1})\\&+\eta_t^2\Vert Z_t\Vert^2_2.
\end{align*}
By plugging the above inequality into \eqref{eq:13}, we get
\begin{align*}
&\eta_t \tr{(X_t-X) F(X)}\leq H({X},Y_{t})-H({X},Y_{t+1})\\&+H({X},U_{t})-H({X},  U_{t+1})+\eta_t^2\Vert   Z_t\Vert^2_2\\
	&+\eta_t \tr{(  V_t-  X_t)   Z_t}+\eta_t^2\Vert \Phi_t\Vert^2_2. 
\end{align*}
By summing the above inequality form $t=0$ to $T-1$, and rearranging the terms, we have
\begin{align*} 
&\sum\nolimits_{t=0}^{T-1}\eta_t \tr{(X_t-X)F(X)}\leq H({X},Y_{0})-H(  {X},  Y_{T})
\end{align*}
\begin{align} 
&+H({X},U_{0})-	H({X},U_{T})+\sum\nolimits_{t=0}^{T-1}\eta_t^{2}\Vert Z_t\Vert^2_2+\nonumber\\ 
&\sum\nolimits_{t=0}^{T-1}\eta_t \tr{(  V_t-  X_t)   Z_t}+\sum\nolimits_{t=0}^{T-1}\eta_t^{2}\Vert \Phi_t\Vert^2_2 \nonumber\\ 
\label{eq:lmerror-4}  
 &\leq H(  {X},  Y_{0})+H(  {X},  U_{0})+\sum\nolimits_{t=0}^{T-1}\eta_t^{2}\Vert   Z_t\Vert^2_2+\nonumber\\&\sum\nolimits_{t=0}^{T-1}\eta_t \tr{(  V_t-  X_t)   Z_t}+\sum\nolimits_{t=0}^{T-1}\eta_t^{2}\Vert \Phi_t\Vert^2_2,	
\end{align}
\nm{where the last inequality holds by $H({X},Y)\geq 0$ \cite{mertikopoulos2017distributed}. By choosing $Y_{0}=U_{0}=\mathbf I_n$ and recalling that  for $X \in \mathcal X$, $\tr{X}=1$ and $-\log(n)\leq\tr{X \log X}\leq 0$ \cite{carlen2010trace}, from \eqref{eq:entropy}, \eqref{eq:conjugate} and \eqref{eq:fenchel},}
\begin{align*}
	&H(  {X},  Y_{0})=H(  {X},  U_{0})=\tr{  X \log   X-  X}-\tr{X}\\&+\log\tr{\exp(2\mathbf I_n)} \leq 0-1-1+\log(n) \leq \log(n).
\end{align*}
Plugging the above inequality into \eqref{eq:lmerror-4} yields
\begin{align}  
\label{eq:bound-5}
&\sum\nolimits_{t=0}^{T-1}\eta_t \tr{(X_t-X) F(X)}=\sum\nolimits_{t=0}^{T-1}\eta_t \text{tr}((  X_t-  X)\nonumber\\&  F(  X)) 
 \leq 2\log(n)+\sum\nolimits_{t=0}^{T-1}\eta_t^{2}\Vert   Z_t\Vert^2_2+\nonumber\\
&\sum\nolimits_{t=0}^{T-1}\eta_t \tr{(  V_t-  X_t)   Z_t} +\sum\nolimits_{t=0}^{T-1}\eta_t^{2}\Vert \Phi_t\Vert^2_2.
\end{align}
Let us define $\gamma_t:\triangleq\frac{\eta_t}{\sum_{k=0}^{T-1} \eta_k}$ and $\overline {  X}_T:\triangleq\sum_{t=0}^{T-1}\gamma_t   X_t$. We divide both sides of \eqref{eq:bound-5} by $\sum_{t=0}^{T-1} {\eta_t}$. Then for all $X\in \mathcal X$,
\begin{align*}  
&\tr{\left(\sum_{t=0}^{T-1} \gamma_t  X_t-X\right)  F(X)} =
\tr{\left(\overline{  X}_T-  X\right)  F(  X)}  \\
&\leq \frac{1}{\sum_{t=0}^{T-1}\eta_t} \Bigg(2\log(n)+\sum_{t=0}^{T-1}\eta_t^{2}\Vert Z_t\Vert^2_2\nonumber \\&+\sum\nolimits_{t=0}^{T-1}\eta_t \tr{(  V_t-  X_t)Z_t}+\sum\nolimits_{t=i}^{T}\eta_t^{2}\Vert \Phi_t\Vert^2_2\Bigg).
\end{align*}
The set $\mathcal X$ is a convex set. Since $\gamma_t>0$ and $\sum_{t=0}^{T-1}\gamma_t=1$,  $\overline{X}_T \in \mathcal X$. 
Now, we take the supremum over the set $\mathcal X$ with respect to $X$ and use the definition of the $G$ function. Note that the right-hand side of the above inequality is independent of $  X$.
\begin{align*}  
&G(\overline{  X}_T) \leq \frac{1}{\sum_{t=0}^{T-1}\eta_t}\Bigg(2\log(n)+\sum_{t=0}^{T-1}\eta_t^{2}\Vert   Z_t \Vert ^2_2+\\
&\sum\nolimits_{t=0}^{T-1}\eta_t \tr{(  V_t-  X_t)   Z_t}+\sum\nolimits_{t=0}^{T-1}\eta_t^{2}\Vert \Phi_t\Vert^2_2 \Bigg).
\end{align*}
By taking expectations on both sides, we get 
\begin{align*}  
&\EX[G(\overline{  X}_T)] \leq \frac{1}{\sum_{t=0}^{T-1}\eta_t}\Bigg(2\log(n)+\sum_{t=0}^{T-1}\eta_t^{2}\EX[\Vert   Z_t \Vert ^2_2]+\\&\sum\nolimits_{t=0}^{T-1}\eta_t \EX[\tr{(  V_t-  X_t)   Z_t}]+\sum\nolimits_{t=0}^{T-1}\eta_t^{2}\EX[\Vert \Phi_t\Vert^2_2] \Bigg).
\end{align*}
By definition, both $X_t$ and $V_t$ are $\mathcal F_t$-measurable. Therefore, $V_t-X_t$ is $\mathcal F_t$-measurable. In addition, $Z_t$ is $\mathcal F_{t+1}$-measurable. Thus, by Assumption \ref{ass:boundphi}(c), we have $\EX[\tr{(  V_t-  X_t)   Z_t}|\mathcal F_t]=0$. Applying Assumption \ref{ass:boundphi}(b), 
\begin{align*}  
\EX[G(\overline{  X}_T)] &\leq  \frac{2}{\sum\nolimits_{t=0}^{T-1}\eta_t}  \left(\log(n)+\sum\nolimits_{t=0}^{T-1}\eta_t^{2} C^2\right). 
\end{align*}
  \end{proof}
Next, we present convergence rate of the AM-SMD scheme.
\begin{thm} 
Consider problem \eqref{eq:VI} and let the sequence $\{\overline{  X}_t\}$ be generated by AM-SMD algorithm. Suppose Assumption \ref{ass:boundphi} holds. Then,
\begin{align}
\label{eq:stepsize}
	&\eta_t=\frac{1}{C}\sqrt{\frac{\log(n)}{T}}, \quad \text{for all} \quad t\geq 0, 
	\\
	\label{eq:rate}
	&\EX[G(\overline{  X}_T)] \leq 3C\sqrt{\frac{\log(n)}{T}}={\cal O}\left(\frac{1}{\sqrt {T}}\right).
\end{align}
\end{thm}
  \begin{proof}
\nm{Consider relation \eqref{eq:bound-1}.} Assume that the number of iterations $T$ is fixed and $\eta_t=\eta$ for all $t \geq 0$, then, we get
\begin{align*}  
\EX[G(\overline{  X}_T)] &\leq \frac{2 (\log(n)+T\eta^{2} C^2)}{T\eta}.
\end{align*}
Then, by minimizing the right-hand side of the above inequality over $\eta>0$, we obtain the constant stepsize \eqref{eq:stepsize}.
By plugging \eqref{eq:stepsize} into \eqref{eq:bound-1}, we obtain \eqref{eq:rate}.
  \end{proof}

\section{Numerical experiments}
\label{sec:num}
In this section, we examine the behavior of the AM-SMD method on the throughput maximization problem in a multi-user MIMO wireless network as described in Section \ref{sec:int}. First, we need to show that the Nash equiblrium of game \eqref{eq:Rgame} is a solution of VI$(\mathcal X,F)$. 
\nm{Since the throughput function $R_i(X_i,X_{-i})$ given by \eqref{eq:game} is a concave function, we can apply Lemma \ref{lemma:nash}.}
\nm{
We have $\nabla_{ X_i}R_i(  X_i,  X_{-i})=H_{ii}^\dagger W^{-1}  H_{ii}$ \cite{mertikopoulos2016learning}. By concavity of $R_i(  X_i,  X_{-i})$ in $X_i$ and convexity of $\mathcal X_i$, the sufficient equilibrium conditions in Lemma \ref{lemma:nash} are satisfied, therefore a Nash equiblrium of game \eqref{eq:Rgame} is a solution of VI \eqref{eq:VI}, where $\mathcal X\triangleq\prod_i \mathcal X_i$ and 
$F(X)\triangleq-\text{diag}\left(  H_{11}^\dagger  W^{-1}  H_{11},\cdots,  H_{NN}^\dagger   W^{-1}  H_{NN}\right)$. Convexity of $-R_i$ results in monotonicity of the mapping $F(X)$. Hence, from Lemma \ref{lemma:convergence}, we have the following corollary. 
}
\begin{col}
The sequence $\overline{X}_t$ generated by AM-SMD algorithm converges to the weak solution of VI$(\mathcal X, F)$.
\end{col}
\subsection{Problem Parameters and Termination Criteria}
We consider a MIMO multicell cellular network composed of seven hexagonal cells (each with a radius of $1$ km) as Figure \ref{fig:cell}. We assume there is one MIMO link (user) in each cell which corresponds to the transmission from a transmitter (T) to a receiver (R). Following \cite{scutari2009mimo}, we generate the channel matrices with a Rayleigh distribution, i.e, each element is generated as circularly symmetric Guassian random variable with a variance equal to the inverse of the square distance between the transmitters and receivers. 
The network can be considered as a 7-users game where each user is a MIMO channel. Distance between receivers and transmitters are shown in Table \ref{table:distance}. 
It should be noted that the channel matrix between any pair of transmitter $i$ and receiver $j$ is a matrix with dimension of $n_j \times m_i$. In the experiments, we assume $n_j=n$ for all $j$ and $m_i=m$ for all $i$. As an example, the channel matrix between transmitter 4 and receiver 5, where $n=m=4$ is represented in Table \ref{table:channel}.  
Moreover, the transmitters have a maximum power of $1$ decibels of the measured power referenced to one milliwatt (dBm). 
\begin{figure}[H]
\begin{center}
  \includegraphics*[scale=0.35]{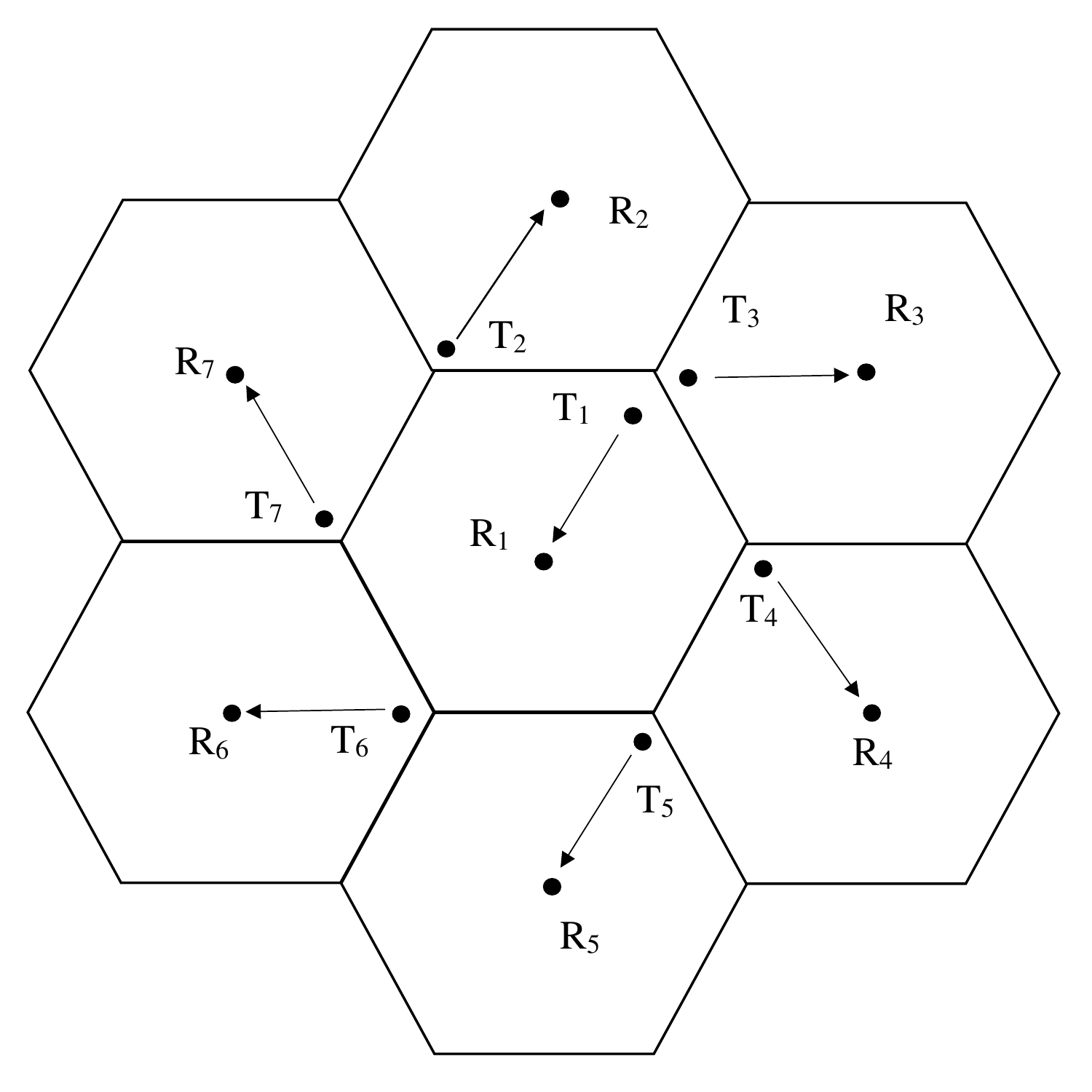}
	\caption{\scriptsize {Multicell cellular system}}
	\label{fig:cell}
	\end{center}
\end{figure}
\begin{table}[ht]
\caption{Distance matrix (in terms of kilometer)}
\label{table:distance}
\begin{center}
\begin{tabular}{|c|c c c c c c c|}
\hline
& R1 & R2 & R3 & R4 & R5 & R6 & R7  \\ \hline
T1 & 0.89 & 1.01 & 1.05 & 1.10 & 1.01 & 1.05 & 1.10  \\
T2 & 1.01 & 0.89 & 1.05 & 2.10 & 2.69 & 2.66 & 1.99	\\
T3 & 1.10 & 1.90 & 0.89 & 1.01 & 2.10 & 2.72 & 2.72	 \\ 
T4 & 1.99 & 2.61 & 1.94 & 0.89 & 1.10 & 2.10 & 2.76  \\
T5 & 2.56 & 2.69 & 2.66 & 1.99 & 0.89 & 1.05 & 2.10 \\
T6 & 2.52 & 2.10 & 2.72 & 2.72 & 1.90 & 0.89 & 1.01 \\
T7 & 1.90 & 1.10 & 2.10 & 2.76 & 2.61 & 1.94 & 0.89 \\
\hline
\end{tabular}
\end{center}
\end{table}
 
\begin{table}[ht]
\caption{Channel matrix between transmitter 4 and receiver 5 (in terms of decibels)}
\label{table:channel}
\begin{center}
\begin{tabular}{|c|c c c c|}
\hline
& RA1 & RA2 & RA3 & RA4  \\ \hline
TA1 & -0.54-0.71i & -1.39+2.24i & 0.65-2.17i & 0.84+0.17i  \\
TA2 & -0.13-0.71i & -0.14+0.88i & 0.09-1.67i & -1.22-0.25i	\\
TA3 & 1.39+2.34i  & -0.17+1.23i & 1.00+0.23i & 1.72-0.33i	 \\ 
TA4 & 2.40-0.97i  & 1.10-1.07i  & 2.94-2.00i & 0.21-1.64i  \\
\hline
\end{tabular}
\end{center}
\end{table}
We investigate the robustness of the AM-SMD algorithm under imperfect feedback. To simulate imperfections, we generate a zero-mean circularly symmetric complex Gaussian noise vector $  Z_t$ with covariance matrix $\sigma \mathbf I_{m}$, where $m=\sum_{i=1}^7 {m_i}$. 
In experiments, we consider the following gap function $Gap(  X)$ which is equal to zero for a strong solution.
\begin{defn} [A gap function] \label{def:gap2} Define the following function $Gap: \mathcal X \rightarrow \mathbb R$
\begin{equation}
\label{gap3}
	Gap(  {X})= \underset {  Z \in \mathcal X} {\sup}\tr{F(  X)(  {X}-  {Z})}, \quad \text{for all}~   {X} \in  \mathcal X.
\end{equation}
\end{defn}
Next, we provide some properties of the Gap function.
\begin{lemma}
\label{lm:gap properties}
The function $Gap(  {X})$ given by Definition \ref{def:gap2} is a well-defined gap function, in other words, $(i)$ $Gap(  {X})$ is nonnegative for all $  X \in \mathcal X$; and $(ii)$ $  X^*$ is a strong solution to problem \eqref{eq:VI} iff $Gap(  {X^*})=0$.
\end{lemma}
The proof is similar to the proof of Lemma \ref{lm:gap properties2}.

The algorithms are run for a fixed number of iterations $T$. We plot the gap function for different number of transmitter and receiver antennas ($m,n$). We also plot the gap function for different values of $\sigma$ including $0.5, 1, 5$. We use MATLAB to run the algorithms and CVX software to solve the optimization problem \eqref{gap3}. 
\subsection{Matrix Exponential Learning} 
Mertikopoulos et al. \cite{mertikopoulos2017distributed} proved the convergence of MEL algorithm under strong monotonicity of mapping $F$ assumption while, in practice, this assumption might not hold for the games and VIs. \nm{We established the convergence of the AM-SMD and derived a rate statement without assuming strong monotonicity. Here, we compare the performance of the AM-SMD method with that of MEL under regularization. }
Doing so, we obtain a strongly monotone mapping (\cite{facchinei2007finite}, Chapter 2).
Let $\Vert A\Vert_F=\sqrt{\tr{A^TA}}$ denote the Frobenius norm of a matrix $A$. 
Note that the function $h(  A)=\frac{1}{2}\Vert   A\Vert_F^2$ is strongly convex with parameter 1 and $\nabla \frac{\lambda}{2} \Vert   X \Vert_F^2=\lambda   X$. Therefore, to regularize the mapping $F$, we need to add the term $\lambda  X$ to it and consequently, the mapping $F'=F+\lambda  X$ is different from the original $F$. Note that for small values of $\lambda$, MEL converges very slowly. On the other hand, the solution which is obtained by using large values of $\lambda$ is far from the solution to the original problem. Hence, we need to find a reasonable value of $\lambda$. For this reason, we tried three different values for $\lambda$ including $0.1, 0.5, 1$.  
\begin{table}[h]
\setlength{\tabcolsep}{3pt}
\centering
 \begin{tabular}{c| c  c  c}
& $\sigma=0.5$ & $\sigma=1$ & $\sigma=5$ \\ \hline\\
\begin{turn}{90}
$(2,4)$
\end{turn}
&
\begin{minipage}{.14\textwidth}
\includegraphics[scale=.17, angle=0]{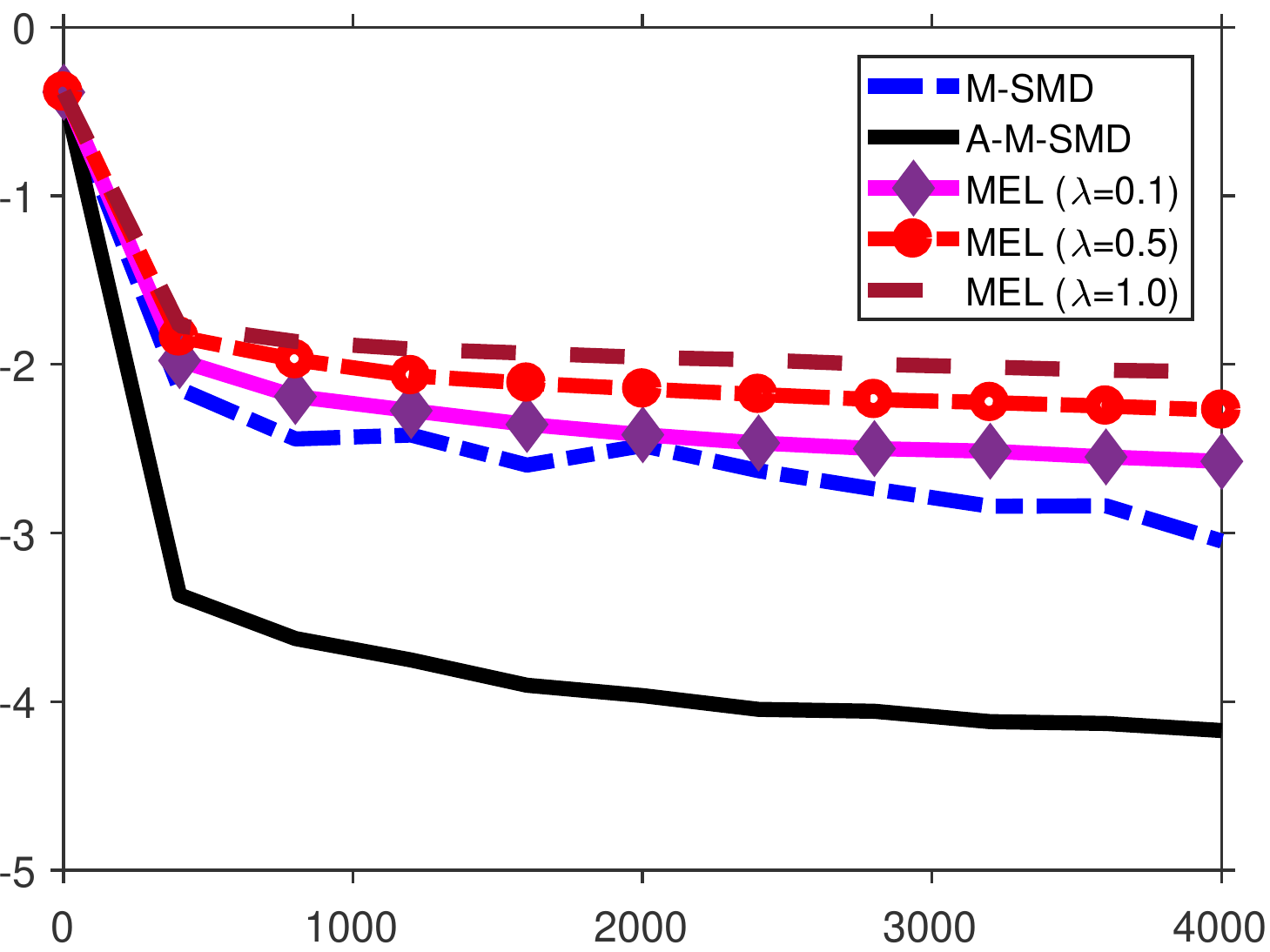}
\end{minipage}
&
\begin{minipage}{.14\textwidth}
\includegraphics[scale=.17, angle=0]{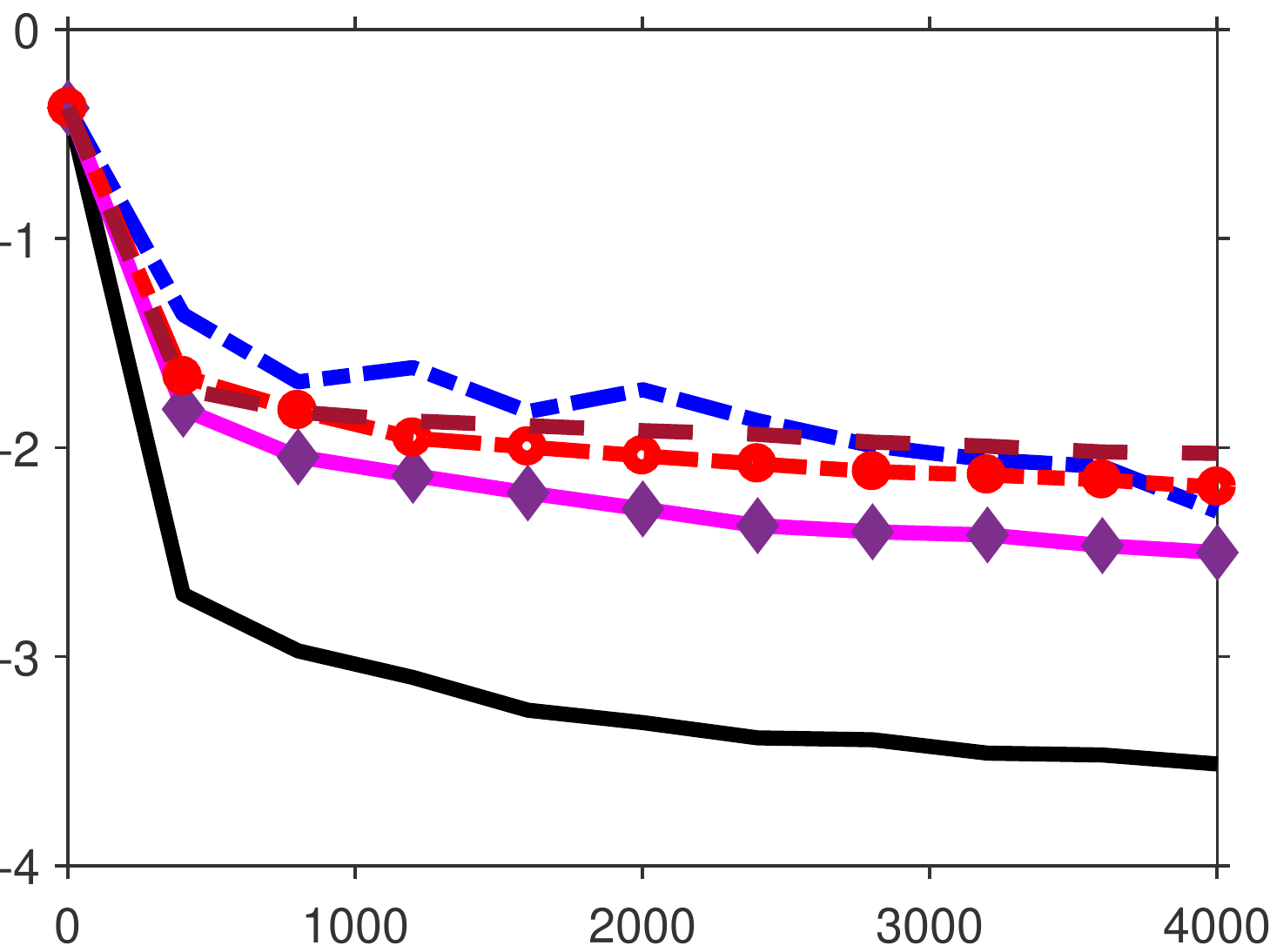}
\end{minipage}
	&
\begin{minipage}{.14\textwidth}
\includegraphics[scale=.17, angle=0]{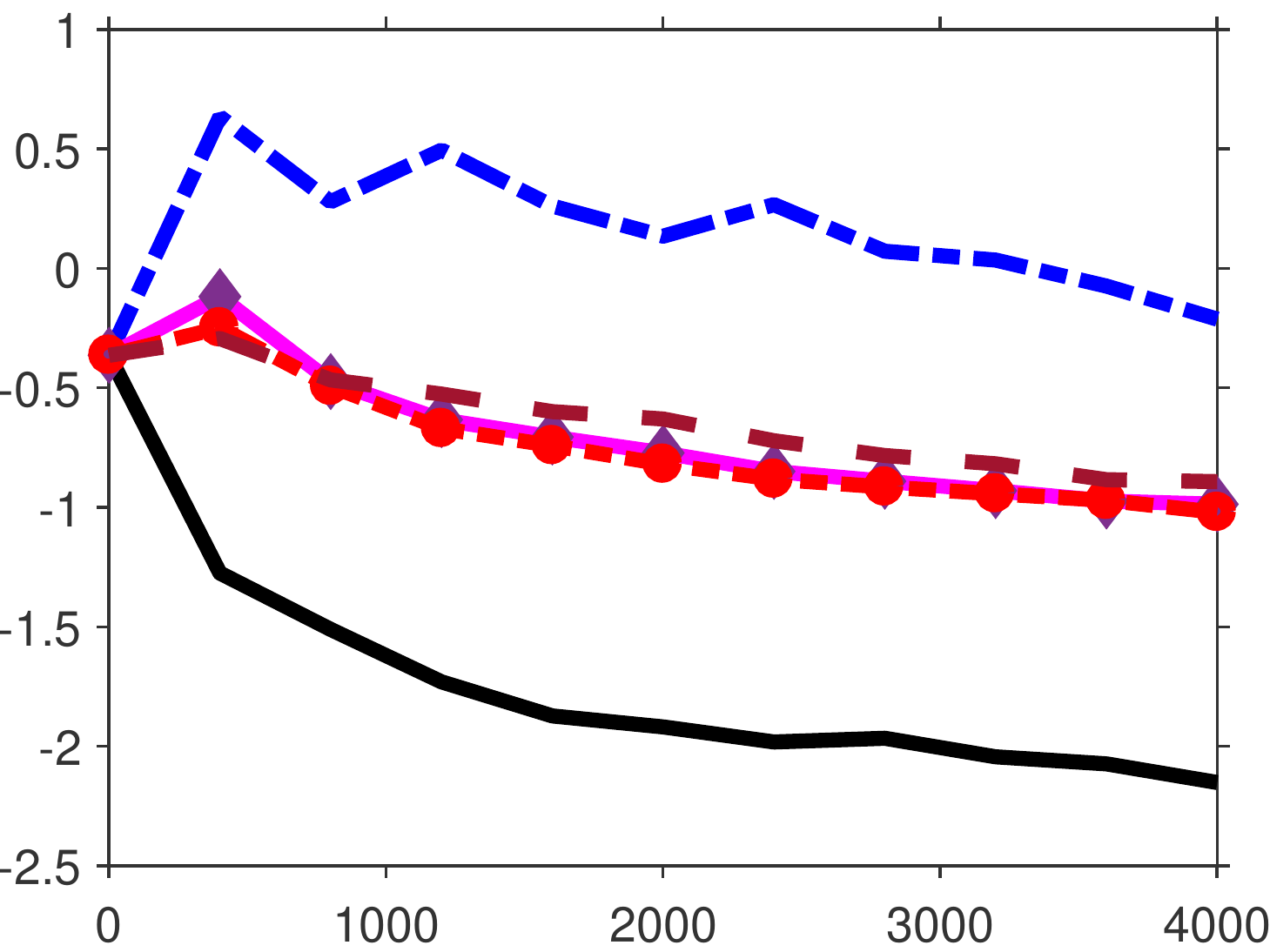}
\end{minipage}
\\
\begin{turn}{90}
$(4,2)$
\end{turn}
&
\begin{minipage}{.14\textwidth}
\includegraphics[scale=.17, angle=0]{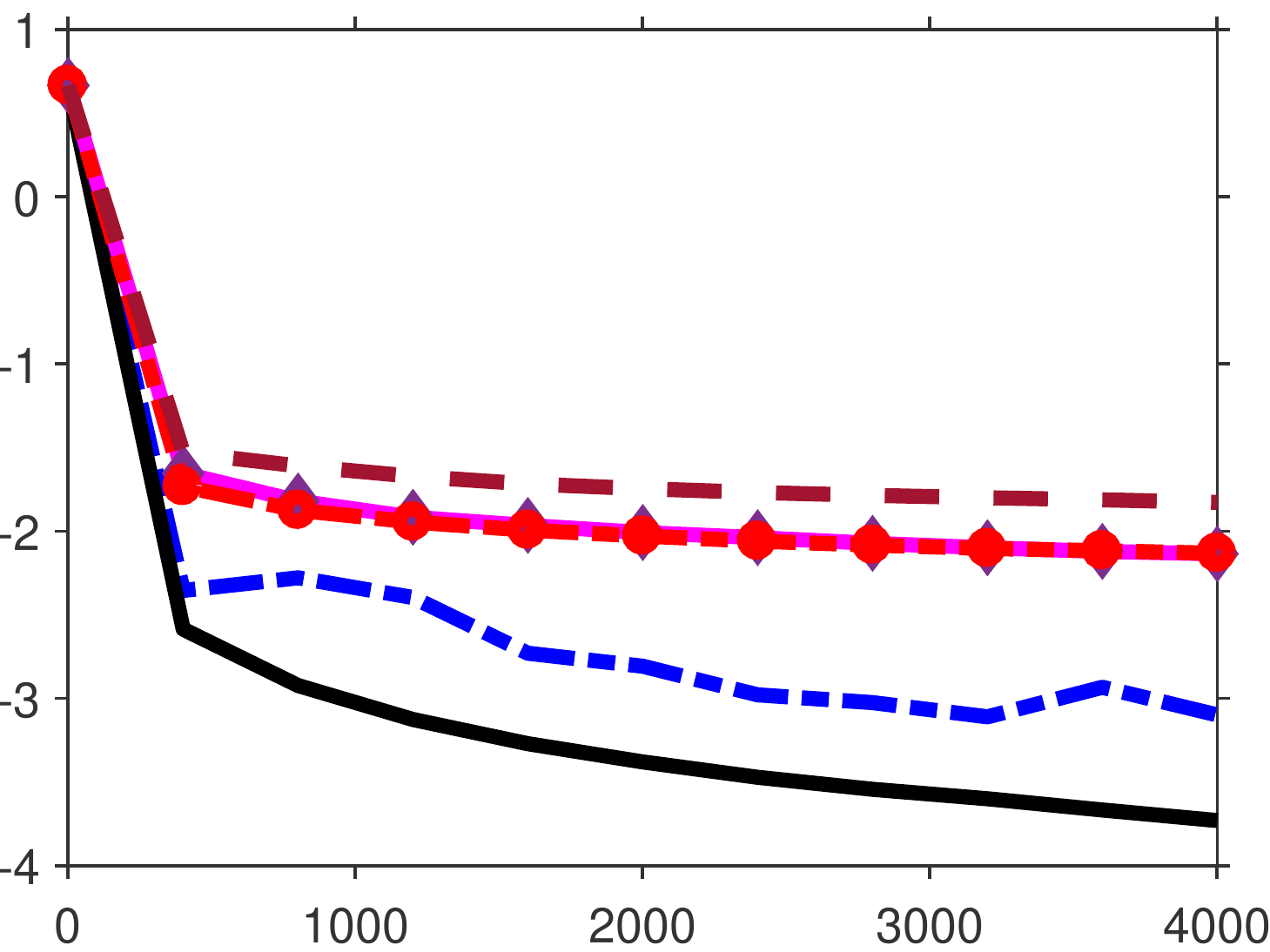}
\end{minipage}
&
\begin{minipage}{.14\textwidth}
\includegraphics[scale=.17, angle=0]{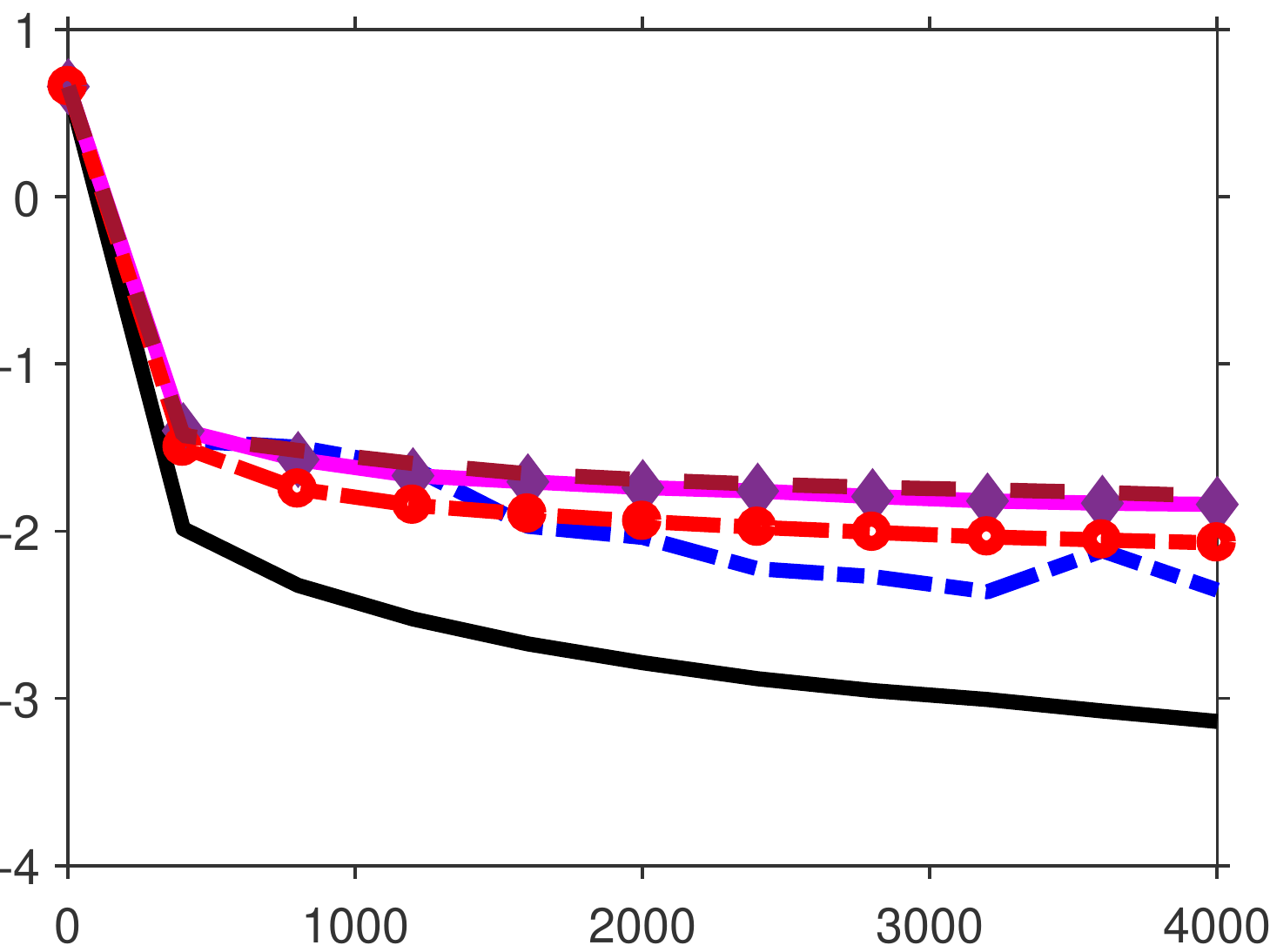}
\end{minipage}
&
\begin{minipage}{.14\textwidth}
\includegraphics[scale=.17, angle=0]{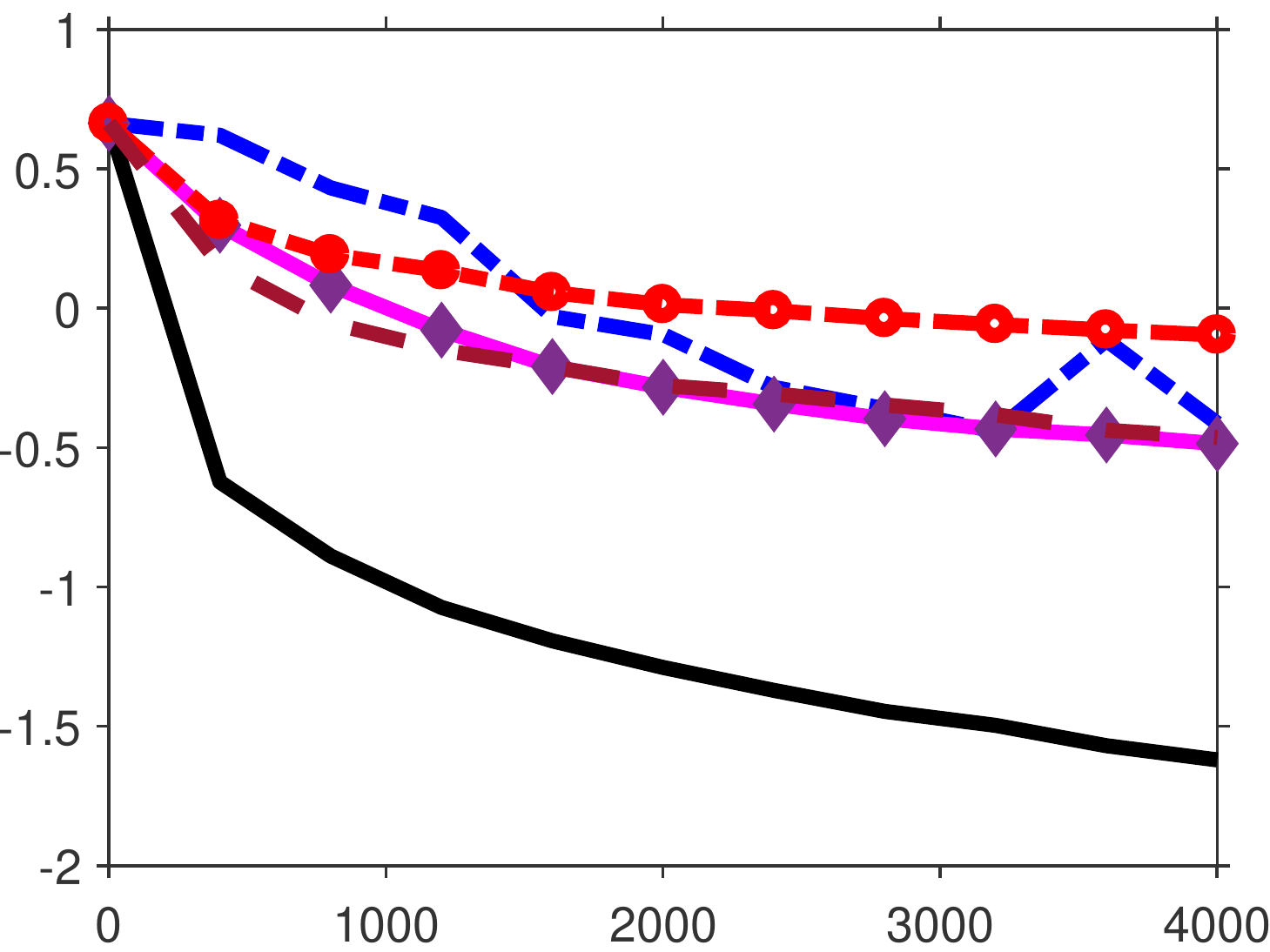}
\end{minipage}
\\
\begin{turn}{90}
$(4,4)$
\end{turn}
&
\begin{minipage}{.14\textwidth}
\includegraphics[scale=.17, angle=0]{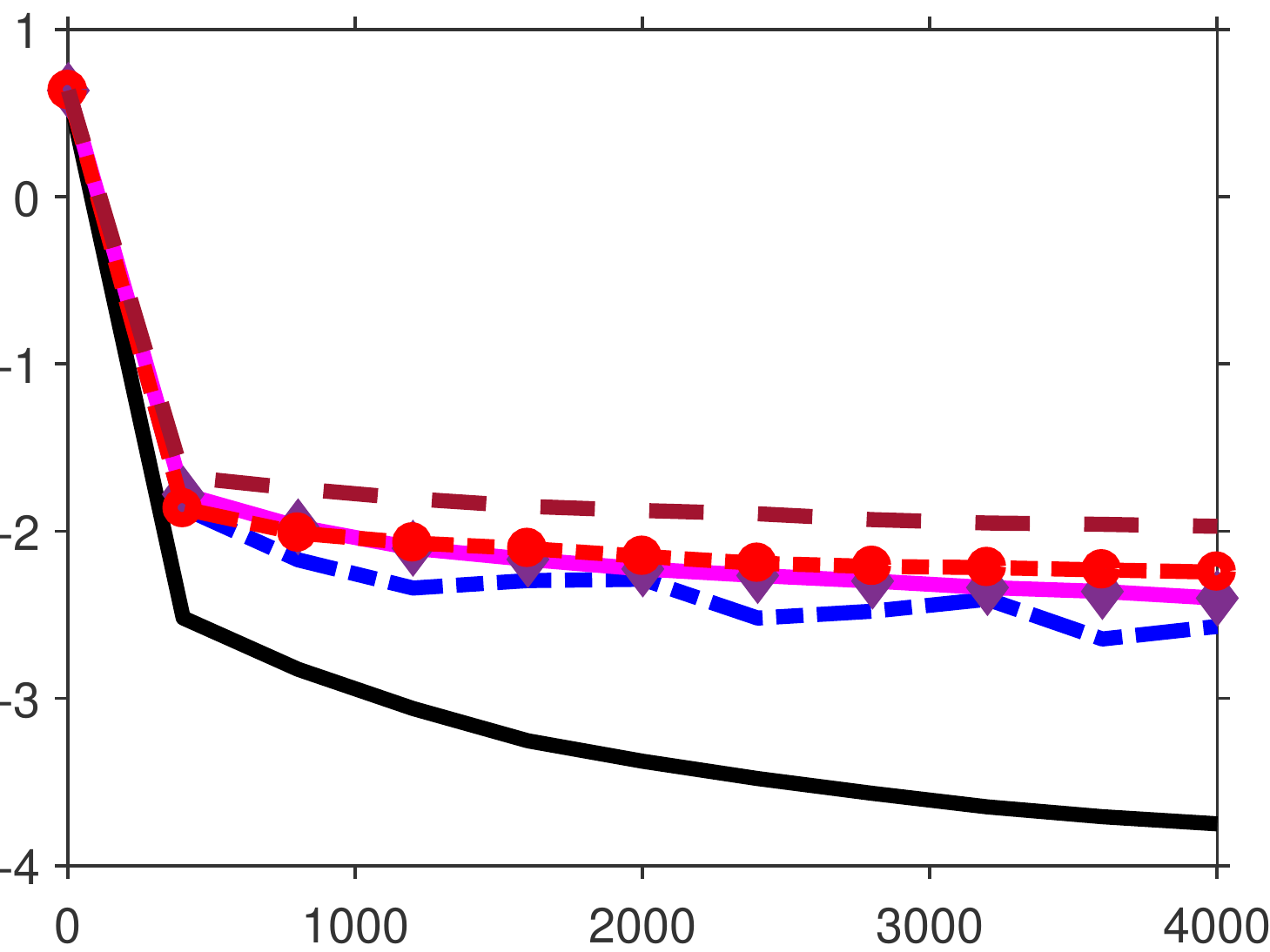}
\end{minipage}
&
\begin{minipage}{.14\textwidth}
\includegraphics[scale=.17, angle=0]{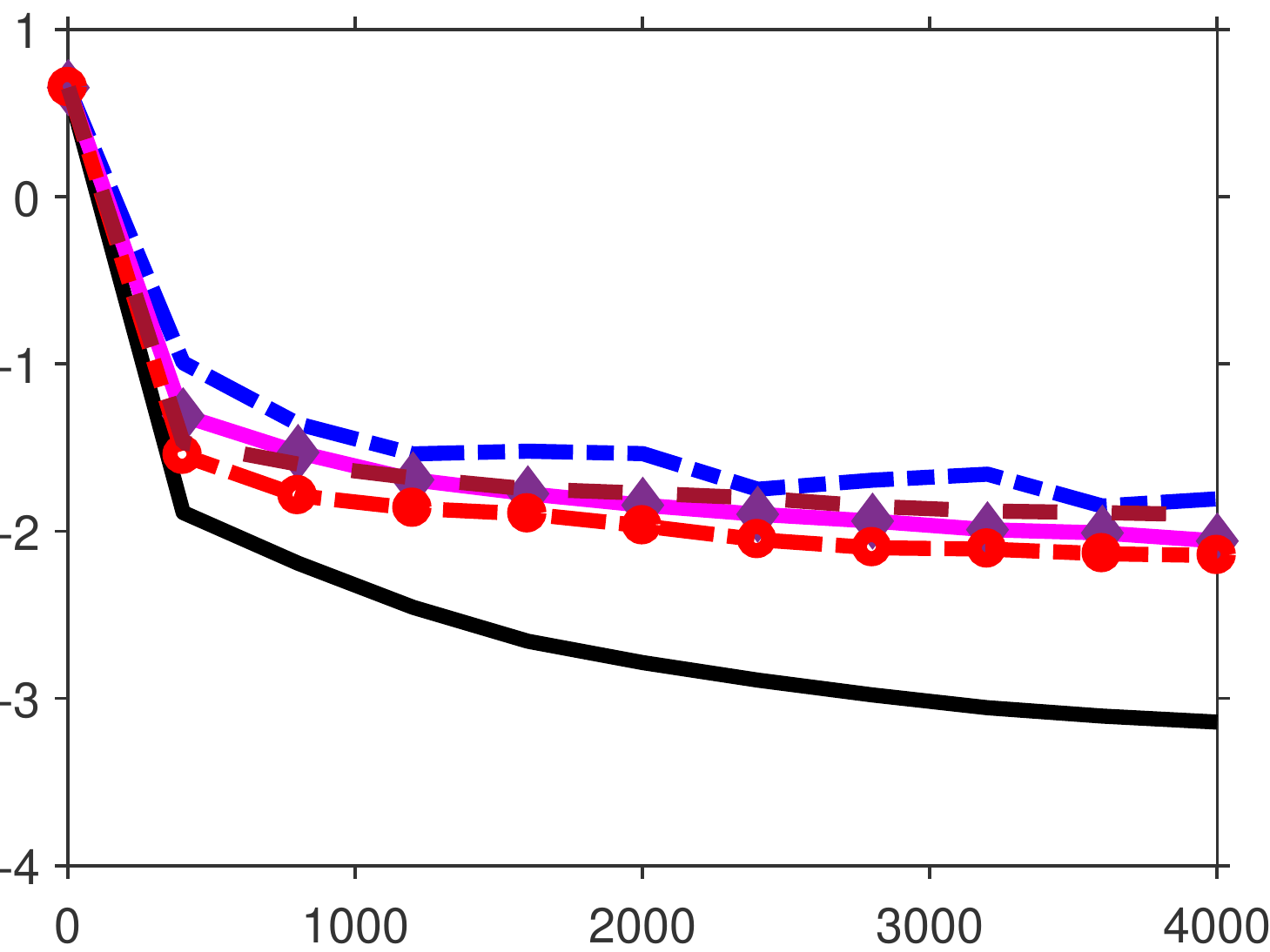}
\end{minipage}
&
\begin{minipage}{.14\textwidth}
\includegraphics[scale=.17, angle=0]{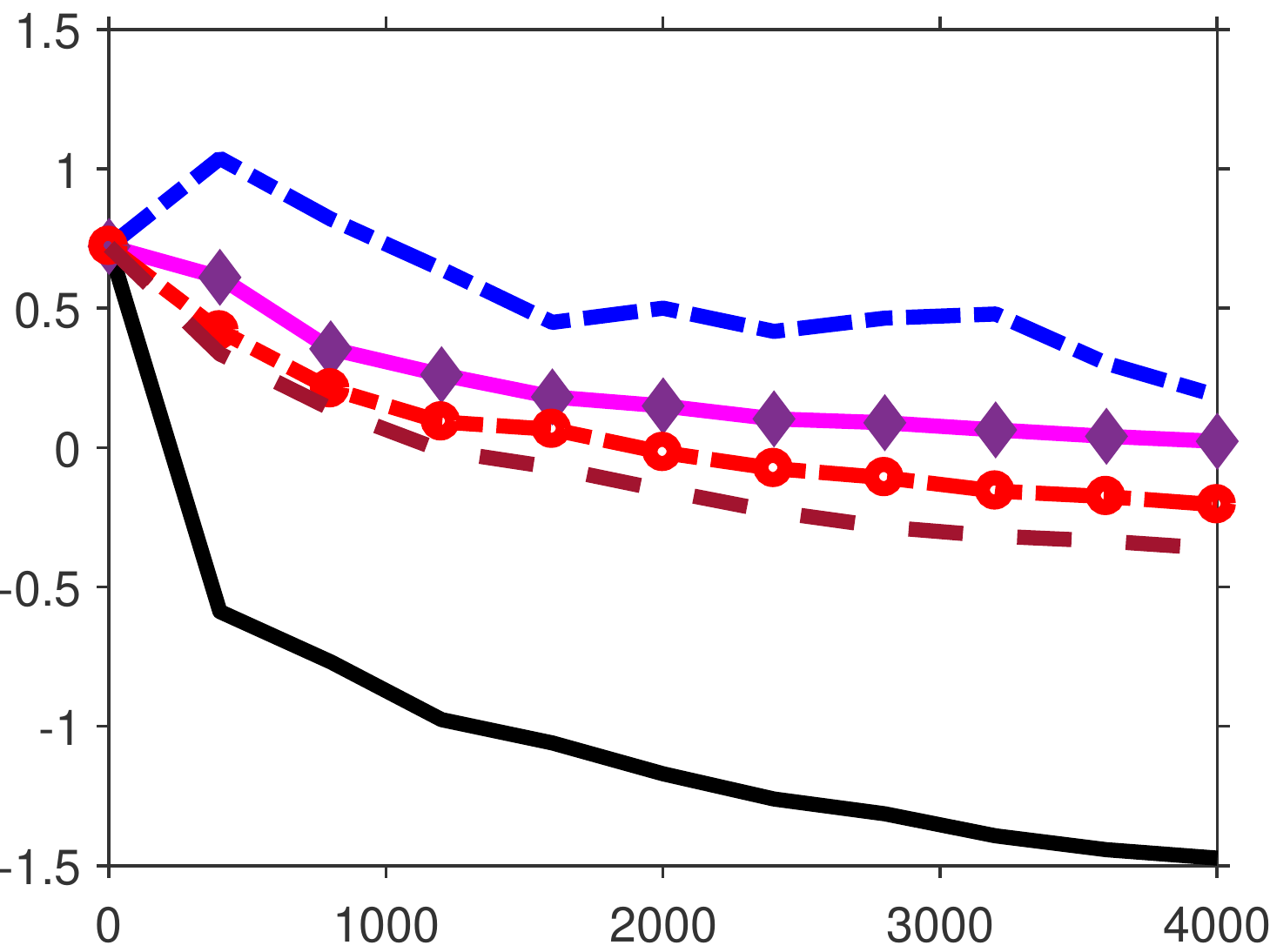}
\end{minipage}
\end{tabular}
\captionof{figure}{\scriptsize{Comparison of M-SMD, AM-SMD and MEL w.r.t. problem size ($m,n$), uncertainty ($\sigma$), and regularization parameter ($\lambda$) for 4000 iterations}}
\label{fig:fiveplots}
\end{table}

For each experiment, the algorithm is run for $4000$ iterations. We apply the well-known harmonic stepsize $\eta_t={1}/{\sqrt{t}}$ for AM-SMD and M-SMD, and harmonic stepsize $\eta_t={1}/{t}$ for MEL. 
Figure \ref{fig:fiveplots} demonstrates the performance of AM-SMD, M-SMD and MEL algorithms in terms of logarithm of expected value of gap function \eqref{gap3}. The expectation is taken over $  Z_t$, we repeat the algorithm for $10$ sample paths and obtain the average of the gap function. In these plots, the blue (dash-dot) and black (solid) curves correspond to the M-SMD and AM-SMD algorithms, respectively, the magenta (solid diamond), red (circle dashed) and brown (dashed) curves display MEL algorithm with $\lambda=0.1, 0.5$ and $1$.
As can be seen in Figure \ref{fig:fiveplots}, AM-SMD algorithm outperforms the M-SMD and MEL algorithms in all experiments. It is evident that MEL algorithm converges slowly but faster than M-SMD. Comparing three versions of MEL algorithm which apply large, moderate or small value of regularization parameter $\lambda$, it can be seen that MEL is not robust w.r.t this parameter since each one of MEL algorithms has a better performance than the other two in some cases.  
\begin{table}[h]
\setlength{\tabcolsep}{3pt}
\centering
 \begin{tabular}{c  c  c}
\begin{minipage}{.15\textwidth}
\includegraphics[scale=.17, angle=0]{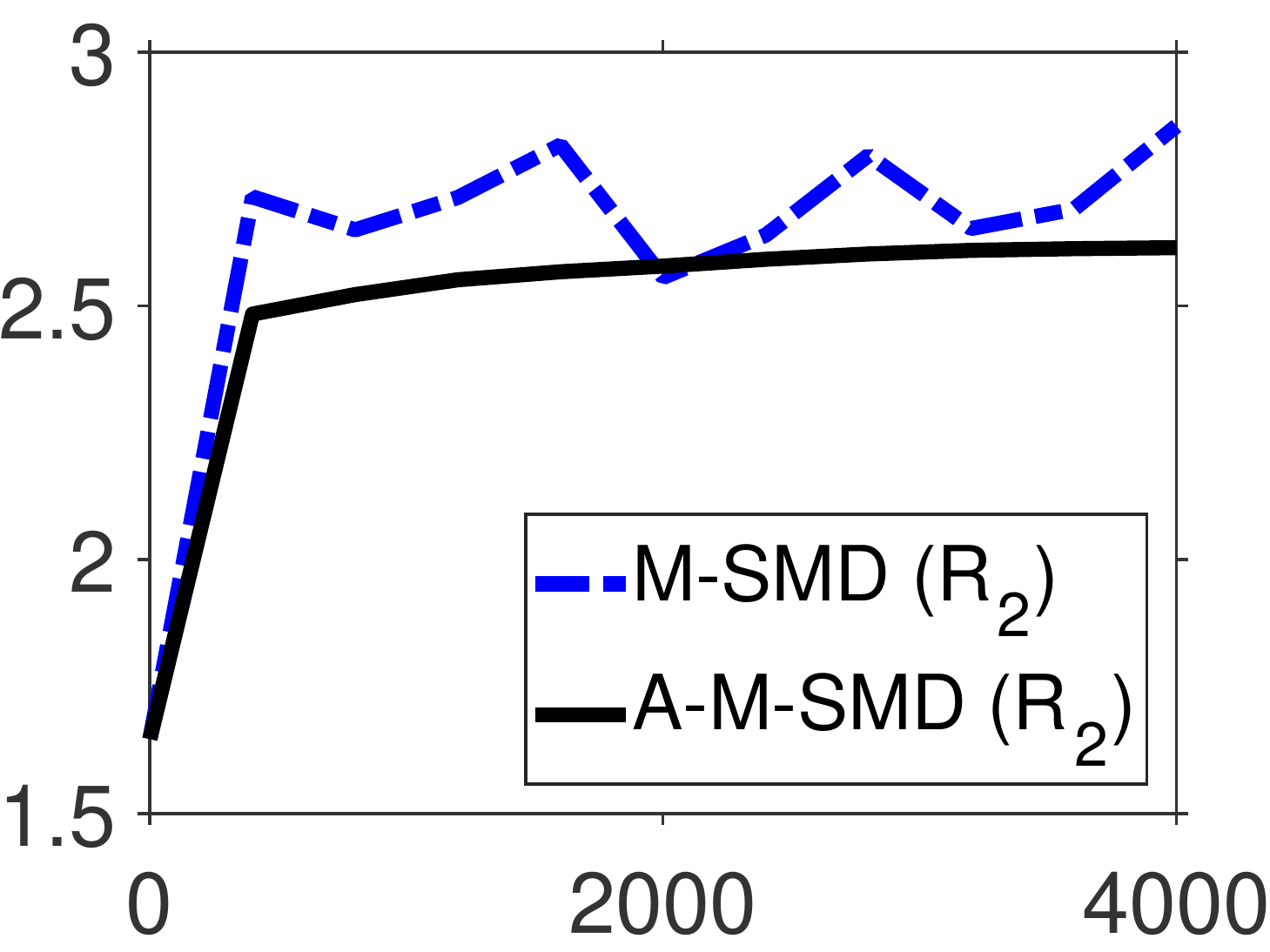}
\end{minipage}
&
\begin{minipage}{.15\textwidth}
\includegraphics[scale=.17, angle=0]{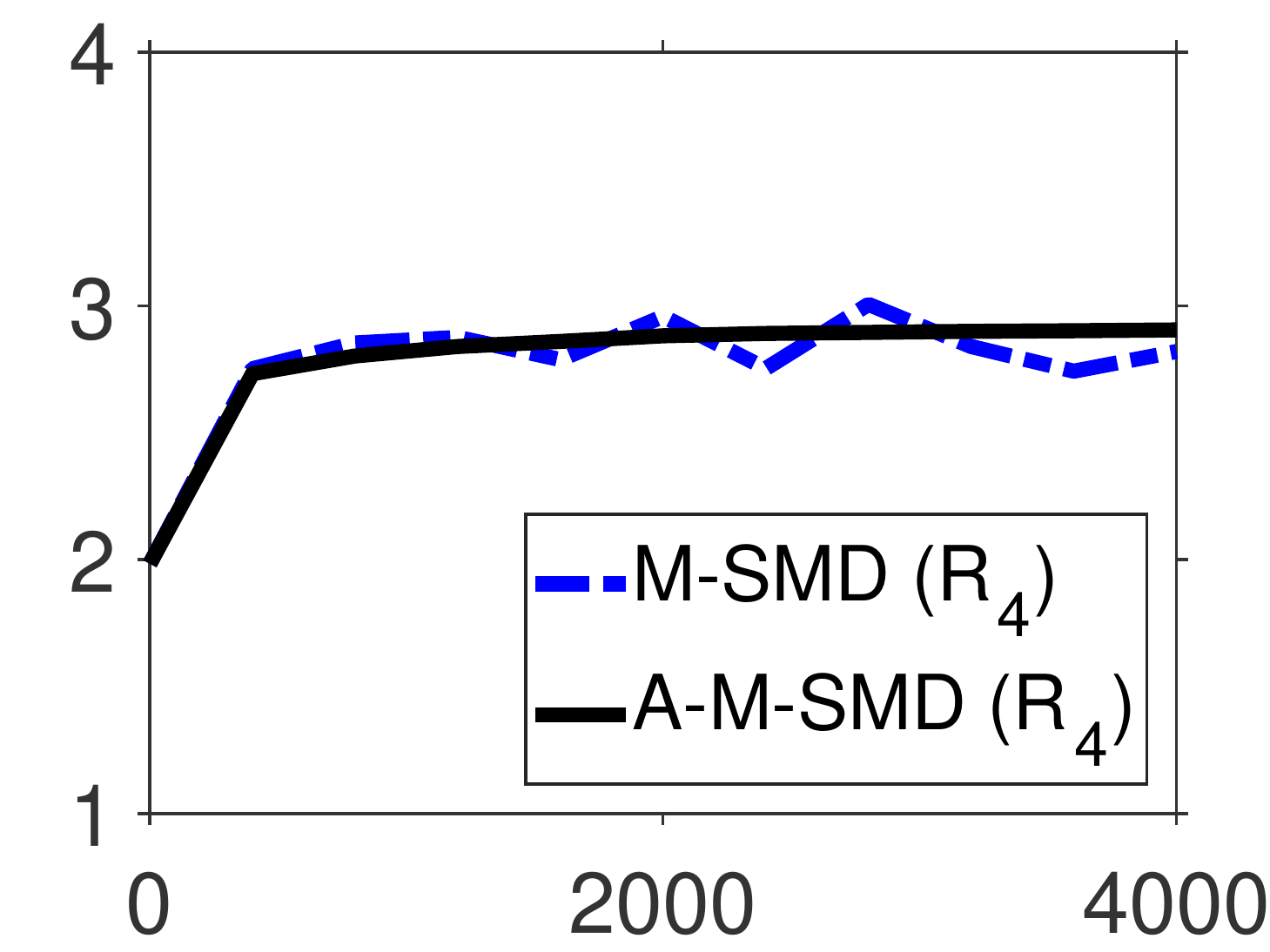}
\end{minipage}
&
\begin{minipage}{.15\textwidth}
\includegraphics[scale=.17, angle=0]{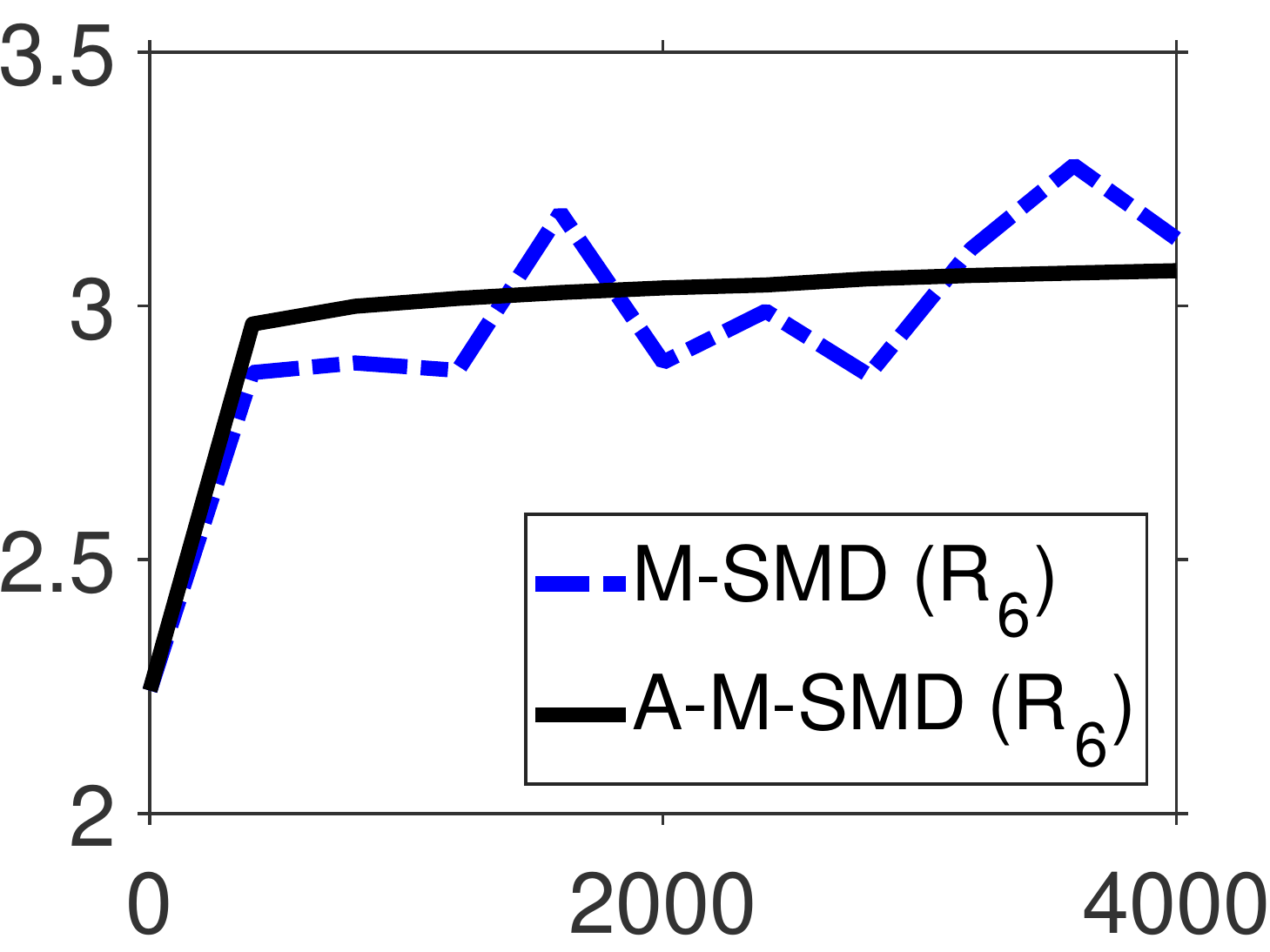}
\end{minipage}
\end{tabular}
\captionof{figure}{\scriptsize {Comparison of stability of M-SMD and AM-SMD in terms of users' objective function $R_i$ for $i=2,4,6$}}
\label{fig:functionR}
\end{table}

To compare the stability of two methods, we also plot the expected objective function value $R_i$ against the iteration number in Figure \ref{fig:functionR}. Here, we choose $n=m=4$ and $\sigma=10$. The algorithm is repeated for $10$ sample paths and the average of objective function is obtained. Each plot represents the performance of both algorithms for one specific player $i$. As an example, the first plot compares the stability of AM-SMD (black solid curve) and M-SMD (blue dash-dot curve) for the user $2$. It can be seen that for all players, the AM-SMD algorithm converges to a strong solution relatively faster while the M-SMD does not converge and oscillates significantly. 
\section{Conclusion}
\label{sec:conclusion}
We consider stochastic variational inequalities on semidefinite matrix spaces, where the mapping is merely monotone. \nm{ We develop a single-loop first-order method called averaging matrix stochastic mirror descent method and prove convergence to a weak solution of the SVI with rate of ${\cal O}(1/\sqrt{T})$}. Our numerical experiments performed on a wireless communication network display that the AM-SMD method is significantly robust w.r.t. the problem size and uncertainty. 

\section{Appendix}
Proof of Lemma \ref{pre:smoothstrong}:
Using the Fenchel coupling definition,
\begin{align}
\label{eq:pre2-2}
	H(  {X},  {Y}+  {Z}) = \omega(  {X})+\omega^*(  {Y}+  Z)-\tr{  {X}(  {Y}+  Z)}.
\end{align}
By strong convexity of $\omega$ w.r.t. trace norm (Lemma \ref{lm:strong}) and using duality between strong convexity and strong smoothness \cite{kakade2009duality}, 
$\omega^*$ is 1-strongly smooth w.r.t. the spectral norm, i.e., $\omega^*(  {Y}+  Z) \leq \omega^*(  {Y})+\tr{  Z \nabla \omega^*(  {Y})}+\Vert   Z\Vert^2_2.$
By plugging this inequality into \eqref{eq:pre2-2} we have
	\begin{align*}
&H(  {X},  {Y}+  {Z}) \leq \omega(  {X}) + \omega^*(  {Y})+ \tr{  Z \nabla \omega^*(  {Y})}\\
&+\Vert   Z\Vert^2_2-\tr{  {X}{Y}}-\tr{  {X}{Z}}\\&=H(  {X},  {Y})+\tr{  Z(\nabla \omega^*(  {Y})-  {X})}+\Vert   Z\Vert^2_2,
	\end{align*}
where in the last relation, we used \eqref{eq:fenchel}. 
\bibliographystyle{IEEEtran}
\bibliography{referenceIEEE} 

\begin{thebibliography}{10}
\providecommand{\url}[1]{#1}
\csname url@samestyle\endcsname
\providecommand{\newblock}{\relax}
\providecommand{\bibinfo}[2]{#2}
\providecommand{\BIBentrySTDinterwordspacing}{\spaceskip=0pt\relax}
\providecommand{\BIBentryALTinterwordstretchfactor}{4}
\providecommand{\BIBentryALTinterwordspacing}{\spaceskip=\fontdimen2\font plus
\BIBentryALTinterwordstretchfactor\fontdimen3\font minus
  \fontdimen4\font\relax}
\providecommand{\BIBforeignlanguage}[2]{{%
\expandafter\ifx\csname l@#1\endcsname\relax
\typeout{** WARNING: IEEEtran.bst: No hyphenation pattern has been}%
\typeout{** loaded for the language `#1'. Using the pattern for}%
\typeout{** the default language instead.}%
\else
\language=\csname l@#1\endcsname
\fi
#2}}
\providecommand{\BIBdecl}{\relax}
\BIBdecl

\bibitem{facchinei2007finite}
F.~Facchinei and J.-S. Pang, \emph{Finite-dimensional variational inequalities
  and complementarity problems}.\hskip 1em plus 0.5em minus 0.4em\relax
  Springer Science \& Business Media, 2007.

\bibitem{scutari2010convex}
G.~Scutari, D.~P. Palomar, F.~Facchinei, and J.-s. Pang, ``Convex optimization,
  game theory, and variational inequality theory,'' \emph{IEEE Signal
  Processing Magazine}, vol.~27, no.~3, pp. 35--49, 2010.

\bibitem{foschini1998limits}
G.~J. Foschini and M.~J. Gans, ``On limits of wireless communications in a
  fading environment when using multiple antennas,'' \emph{Wireless personal
  communications}, vol.~6, no.~3, pp. 311--335, 1998.

\bibitem{mertikopoulos2017distributed}
P.~Mertikopoulos, E.~V. Belmega, R.~Negrel, and L.~Sanguinetti, ``Distributed
  stochastic optimization via matrix exponential learning,'' \emph{IEEE
  Transactions on Signal Processing}, vol.~65, no.~9, pp. 2277--2290, 2017.

\bibitem{mertikopoulos2016learning}
P.~Mertikopoulos and A.~L. Moustakas, ``Learning in an uncertain world: {MIMO}
  covariance matrix optimization with imperfect feedback,'' \emph{IEEE
  Transactions on Signal Processing}, vol.~64, no.~1, pp. 5--18, 2016.

\bibitem{telatar1999capacity}
E.~Telatar, ``Capacity of multi-antenna {G}aussian channels,''
  \emph{Transactions on Emerging Telecommunications Technologies}, vol.~10,
  no.~6, pp. 585--595, 1999.

\bibitem{jiang2008stochastic}
H.~Jiang and H.~Xu, ``Stochastic approximation approaches to the stochastic
  variational inequality problem,'' \emph{IEEE Transactions on Automatic
  Control}, vol.~53, no.~6, pp. 1462--1475, 2008.

\bibitem{juditsky2011solving}
A.~Juditsky, A.~Nemirovski, and C.~Tauvel, ``Solving variational inequalities
  with stochastic mirror-prox algorithm,'' \emph{Stochastic Systems}, vol.~1,
  no.~1, pp. 17--58, 2011.

\bibitem{lan2011primal}
G.~Lan, Z.~Lu, and R.~D. Monteiro, ``Primal-dual first-order methods with
  ${O}(1/\epsilon)$ iteration-complexity for cone programming,''
  \emph{Mathematical Programming}, vol. 126, no.~1, pp. 1--29, 2011.

\bibitem{mertikopoulos2012matrix}
P.~Mertikopoulos, E.~V. Belmega, and A.~L. Moustakas, ``Matrix exponential
  learning: Distributed optimization in {MIMO} systems,'' in \emph{Information
  Theory Proceedings (ISIT), 2012 IEEE International Symposium on}.\hskip 1em
  plus 0.5em minus 0.4em\relax IEEE, 2012, pp. 3028--3032.

\bibitem{hsieh2013big}
C.-J. Hsieh, M.~A. Sustik, I.~S. Dhillon, P.~K. Ravikumar, and R.~Poldrack,
  ``{BIG} \& {QUIC}: Sparse inverse covariance estimation for a million
  variables,'' in \emph{Advances in neural information processing systems},
  2013, pp. 3165--3173.

\bibitem{koshal2013regularized}
J.~Koshal, A.~Nedi{\'c}, and U.~V. Shanbhag, ``Regularized iterative stochastic
  approximation methods for stochastic variational inequality problems,''
  \emph{IEEE Transactions on Automatic Control}, vol.~58, no.~3, pp. 594--609,
  2013.

\bibitem{yousefian2016stochastic}
\BIBentryALTinterwordspacing
F.~Yousefian, A.~Nedi{\'c}, and U.~V. Shanbhag, ``On stochastic mirror-prox
  algorithms for stochastic {C}artesian variational inequalities: randomized
  block coordinate and optimal averaging schemes,'' \emph{Set-Valued and
  Variational Analysis}, pp. 1--31, Mar 2018. [Online]. Available:
  \url{https://doi.org/10.1007/s11228-018-0472-9}
\BIBentrySTDinterwordspacing

\bibitem{yousefian2017smoothing}
------, ``On smoothing, regularization, and averaging in stochastic
  approximation methods for stochastic variational inequality problems,''
  \emph{Mathematical Programming}, vol. 165, no.~1, pp. 391--431, 2017.

\bibitem{necoara2017complexity}
I.~Necoara, A.~Patrascu, and F.~Glineur, ``Complexity of first-order inexact
  {L}agrangian and penalty methods for conic convex programming,''
  \emph{Optimization Methods and Software}, pp. 1--31, 2017.

\bibitem{robbins1951stochastic}
H.~Robbins and S.~Monro, ``A stochastic approximation method,'' \emph{The
  Annals of Mathematical Statistics}, pp. 400--407, 1951.

\bibitem{nemirovski2009robust}
A.~Nemirovski, A.~Juditsky, G.~Lan, and A.~Shapiro, ``Robust stochastic
  approximation approach to stochastic programming,'' \emph{SIAM Journal on
  optimization}, vol.~19, no.~4, pp. 1574--1609, 2009.

\bibitem{majlesinasab2017optimal}
N.~Majlesinasab, F.~Yousefian, and A.~Pourhabib, ``Optimal stochastic mirror
  descent methods for smooth, nonsmooth, and high-dimensional stochastic
  optimization,'' \emph{arXiv preprint arXiv:1709.08308v2}, 2017.

\bibitem{korpelevich1977extragradient}
G.~Korpelevich, ``Extragradient method for finding saddle points and other
  problems,'' \emph{Matekon}, vol.~13, no.~4, pp. 35--49, 1977.

\bibitem{dang2015convergence}
C.~D. Dang and G.~Lan, ``On the convergence properties of non-{E}uclidean
  extragradient methods for variational inequalities with generalized monotone
  operators,'' \emph{Computational Optimization and applications}, vol.~60,
  no.~2, pp. 277--310, 2015.

\bibitem{Polyak92}
B.~T. Polyak and A.~B. Juditsky, ``Acceleration of stochastic approximation by
  averaging,'' \emph{SIAM Journal on Control and Optimization}, vol.~30, no.~4,
  pp. 838--855, 1992.

\bibitem{yu2016dynamic}
H.~Yu and M.~J. Neely, ``Dynamic power allocation in mimo fading systems
  without channel distribution information,'' in \emph{International Conference
  on Computer Communications}.\hskip 1em plus 0.5em minus 0.4em\relax IEEE,
  2016, pp. 1--9.

\bibitem{lu2010adaptive}
Z.~Lu, ``Adaptive first-order methods for general sparse inverse covariance
  selection,'' \emph{SIAM Journal on Matrix Analysis and Applications},
  vol.~31, no.~4, pp. 2000--2016, 2010.

\bibitem{tsuda2005matrix}
K.~Tsuda, G.~R{\"a}tsch, and M.~K. Warmuth, ``Matrix exponentiated gradient
  updates for on-line learning and {B}regman projection,'' \emph{Journal of
  Machine Learning Research}, vol.~6, no. Jun, pp. 995--1018, 2005.

\bibitem{fazel2001rank}
M.~Fazel, H.~Hindi, and S.~P. Boyd, ``A rank minimization heuristic with
  application to minimum order system approximation,'' in \emph{Proceedings of
  the American Control Conference}, vol.~6.\hskip 1em plus 0.5em minus
  0.4em\relax IEEE, 2001, pp. 4734--4739.

\bibitem{vedral2002role}
V.~Vedral, ``The role of relative entropy in quantum information theory,''
  \emph{Reviews of Modern Physics}, vol.~74, no.~1, p. 197, 2002.

\bibitem{yu2013strong}
Y.-L. Yu, ``The strong convexity of von {N}eumann's entropy,'' 2013.

\bibitem{bertsekas2009convex}
D.~P. Bertsekas, \emph{Convex optimization theory}.\hskip 1em plus 0.5em minus
  0.4em\relax Athena Scientific Belmont, 2009.

\bibitem{hiai2014introduction}
F.~Hiai and D.~Petz, \emph{Introduction to matrix analysis and
  applications}.\hskip 1em plus 0.5em minus 0.4em\relax Springer Science \&
  Business Media, 2014.

\bibitem{athans1965gradient}
M.~Athans and F.~C. Schweppe, ``Gradient matrices and matrix calculations,''
  Massachusetts Inst of Tech Lexington Lab, Tech. Rep., 1965.

\bibitem{mertikopoulos2016learning2}
P.~Mertikopoulos and W.~H. Sandholm, ``Learning in games via reinforcement and
  regularization,'' \emph{Mathematics of Operations Research}, vol.~41, no.~4,
  pp. 1297--1324, 2016.

\bibitem{carlen2010trace}
E.~Carlen, ``Trace inequalities and quantum entropy: an introductory course,''
  \emph{Entropy and the Quantum}, vol. 529, pp. 73--140, 2010.

\bibitem{scutari2009mimo}
G.~Scutari, D.~P. Palomar, and S.~Barbarossa, ``The {MIMO} iterative
  waterfilling algorithm,'' \emph{IEEE Transactions on Signal Processing},
  vol.~57, no.~5, pp. 1917--1935, 2009.

\bibitem{kakade2009duality}
S.~Kakade, S.~Shalev-Shwartz, and A.~Tewari, ``On the duality of strong
  convexity and strong smoothness: Learning applications and matrix
  regularization,'' \emph{Unpublished Manuscript, http://ttic. uchicago.
  edu/shai/papers/KakadeShalevTewari09. pdf}, 2009.

\end{thebibliography}
\end{document}